 \documentclass[a4paper,12pt,reqno]{amsart}
 \usepackage[french,british]{babel}
 \usepackage{a4wide}
 \usepackage{amsthm}
 \usepackage{amssymb}
 \usepackage{longtable}

 \newtheorem{thm}{Theorem}[section]
 \newtheorem{cor}[thm]{Corollary}
 \newtheorem{prop}[thm]{Proposition}
 \newtheorem{lem}[thm]{Lemma}

 \theoremstyle{definition}

 \theoremstyle{remark}
 \newtheorem{rem}[thm]{Remark}

\makeatletter
\let\c@equation\c@thm
\makeatother
\numberwithin{equation}{section}

\begin{document}
 	
\author{S.Senthamarai Kannan and Pinakinath Saha}
 	
\address{Chennai Mathematical Institute, Plot H1, SIPCOT IT Park, 
 		Siruseri, Kelambakkam,  603103, India.}
 	
\email{kannan@cmi.ac.in.}
 	
\address{Chennai Mathematical Institute, Plot H1, SIPCOT IT Park, 
 		Siruseri, Kelambakkam, 603103, India.}
\email{pinakinath@cmi.ac.in.}
 
\title{rigidity of bott-samelson-demazure-hansen variety for $PSO(2n+1, \mathbb{C})$}

\begin{abstract} Let $G=PSO(2n+1, \mathbb{C}) (n \ge 3)$ and $B$ be the Borel subgroup of $G$ containing maximal torus $T$ of $G.$ Let $w$ be an element of Weyl group $W$ and $X(w)$ be the Schubert variety in the flag variety $G/B$ corresponding to $w.$ Let $Z(w, \underline{i})$ be the Bott-Samelson-Demazure-Hansen variety (the desingularization of $X(w)$) corresponding to a reduced expression $\underline{i}$ of $w.$
 		
 		In this article, we study the cohomology modules of the tangent bundle on $Z(w_{0}, \underline{i}),$ where $w_{0}$ is the longest element of the Weyl group $W.$ We describe all the reduced expressions of $w_{0}$ in terms of a Coxeter element such that all the higher cohomology modules of the tangent bundle on $Z(w_{0}, \underline{i})$ vanish (see Theorem \ref{theorem 8.1}).  
\end{abstract}	
\maketitle
 	
\section{introduction}

 Let $G$ be a simple algebraic group of adjoint type over  the field $\mathbb{C}$ of 
complex numbers. We fix a maximal torus $T$ of $G$ and 
let $W = N_G(T)/T$ denote the Weyl group of $G$ with respect to $T$.
We denote by $R$ the set of roots of $G$ with respect to $T$ and by
 $R^{+}$ a set of positive roots. Let $B^{+}$ be the Borel 
subgroup of $G$ containing $T$ with respect to $R^{+}$. 
Let $w_0$ denote the longest element of the Weyl group $W$. Let $B$ be the 
Borel subgroup of $G$ opposite to $B^+$ determined by $T$, i.e. $B=n_{w_0}B^+n_{w_0}^{-1}$, where $n_{w_0}$ is a representative of $w_0$ in $N_G(T)$.
Note that the 
roots of 
$B$ is the set $R^{-} :=-R^{+}$ of negative roots. We use the notation $\beta<0$ for 
$\beta \in R^{-}$.
Let $S = \{\alpha_1,\ldots,\alpha_n\}$ 
denote the set of all simple roots in $R^{+}$, where $n$ is the rank of $G$. 
The simple reflection in the Weyl group corresponding to a simple root $\alpha$ is denoted
by $s_{\alpha}$. For simplicity of notation, the simple reflection corresponding to a simple root $\alpha_{i}$ is denoted
by $s_{i}$. 

For $w \in W$, let $X(w):=\overline{BwB/B}$ denote the Schubert variety in
$G/B$ corresponding to $w$. Given a reduced expression 
$w=s_{{i_1}}s_{{i_2}}\cdots s_{{i_r}}$ of $w$, with 
the corresponding tuple $\underline i:=(i_1,\ldots,i_r)$, we denote by 
$Z(w,{\underline i})$ the desingularization  of the Schubert variety $X(w)$, 
which is now known as Bott-Samelson-Demazure-Hansen variety. This was 
first introduced by Bott and Samelson in a differential geometric and 
topological context (see \cite{BS}). Demazure in  \cite{De} and 
Hansen in \cite{Han} independently adapted the construction in 
algebro-geometric situation, which explains the reason for the name. 
For the sake of simplicity, we will denote any Bott-Samelson-Demazure-Hansen 
variety by a BSDH-variety. 

The construction of the BSDH-variety $Z(w,{\underline i})$ depends on the 
choice of the reduced expression $\underline i$ of $w$.  In [5],  the automorphism groups of these varieties were studied. There, the following vanishing results of the tangent bundle $T_{Z(w, \underline i)}$ on $Z(w, \underline i)$ were proved 
(see \cite[Section 3]{CKP}):\\
(1) $H^j(Z(w, \underline i), T_{Z(w, \underline i)})=0$ for all $j\geq 2$.\\
(2) If $G$ is simply laced, then $H^j(Z(w, \underline i), T_{Z(w, \underline i)})=0$ for all
  $j\geq 1$.

As a consequence, it follows that the BSDH-varieties are rigid 
for simply laced groups and their deformations are  unobstructed in general (see [5, Section 3] ). 
 The above vanishing result is independent of the choice of the reduced expression $\underline i$ of 
 $w$. While computing the first cohomology module $H^1(Z(w, \underline i),  T_{Z(w, \underline i)})$ for non simply laced group, we observed that this cohomology module very much depend on the choice of a reduced expression $\underline i$ of $w$.

It is a natural question to ask  that for which reduced expressions $\underline i$ of $w$, 
the cohomology module $H^1(Z(w, \underline i), T_{Z(w, \underline i)})$  does vanish ?
In \cite{CK}, a partial answer is given to this question for $w=w_0$ when $G=PSp(2n, \mathbb C)$. 
In this article, we give a partial answer to this question for $w=w_0$ when $G=PSO(2n+1, \mathbb C)$.     

Recall that a Coxeter element is an element of the Weyl group 
having a reduced expression of the form $s_{i_{1}}s_{i_{2}} \cdots s_{i_{n}}$
such that $i_{j}\neq i_{l}$ whenever $j\neq l$ (see \cite[p.56, Section 4.4]{Hum3}). Note that for any Coxeter element $c$, there is a decreasing sequence of integers $n\geq a_1> a_2 > \ldots > a_k=1$ such that $c=\prod\limits_{j=1}^{k}[a_j, a_{j-1}-1]$, where $a_0:=n+1$, $[i, j]:=s_{i}s_{i+1}\cdots s_{j}$ for $i\leq j$.
 
In this paper we prove the following theorem. 
\begin{thm} 
Let $G=PSO(2n+1, \mathbb C)$ $(n\geq 3)$ and let $c\in W$ be a Coxeter element.
Let $\underline i=(\underline i^1, \underline i^2,\ldots, \underline i^n)$ be a sequence corresponding to a reduced expression of $w_0$, where
$\underline i^r$ $(1\leq r\leq n)$ is a sequence of reduced expressions of $c$ (see Lemma 2.8). 
Then, $H^j(Z(w_0, \underline i), T_{(w_0, \underline i)})=0$ for all $j\geq 1$ if and only if $c=\prod\limits_{j=1}^{k}[a_j, a_{j-1}-1]$, where $a_0:=n+1$ and $a_2\neq n-1$.
\end{thm}
 
By the above vanishing results, we conclude that if $G=PSO(2n+1, \mathbb C)$ $(n\geq 3)$ and $\underline i=(\underline i^1, \underline i^2,\ldots, \underline i^n)$
is a reduced expression of $w_0$ as above, then the BSDH-variety $Z(w_0, \underline i)$ is rigid.

The main differences in the proof of main theorem between the case of type $C_{n}$ and the case of type $B_{n}$ are as follows:

 When $G$ is of type $C_{n},$  there is only one long simple root, namely 
$\alpha_{n}.$ Therefore,  by  \cite[Corollary 5.6, p.778]{Ka}, we have  $H^1(w, \alpha_{j})=0$ for any $w\in W$ and for any $j\neq n.$ But, if $G$ is of type $B_{n},$ $\alpha_{1}, \alpha_{2}, \ldots , \alpha_{n-1}$ are all long simple roots. So, we can not apply  \cite[Corollary 5.6, p.778]{Ka}. Hence, we need to study the cohomology modules $H^1(w, \alpha_{j})$ ( $w\in W$, $ j \neq n-1$ ). 
 While studying these modules, we prove that $H^1(w, \alpha_{j})=0$ for any $w\in W$ and for 
 any $j\neq n-1$  (see Lemma \ref{lemma 2.5}).  Also, in this article, we need to prove an additional statement namely,  $(s_{1}s_{2}\cdots s_{n})^{r-1}s_{1}s_{2}\cdots s_{n-1}(\alpha_{j})<0$ for $2\le r \le n$ and $n+1-r\le j\le n-1$ (see Lemma \ref{lemma 3.3}).

 The organization of the paper is as follows:

In Section 2, we recall some preliminaries on BSDH-varieties. We deal with the special case $G=PSO(2n+1, \mathbb C)$  $(n\geq 3)$ in the later sections 3, 4, 5, 6 and 7. In Section 3, we prove $H^1(w, \alpha_{j})=0$ for $j\neq n-1$ and $w\in W.$
In Section 4 (respectively, Section 5) we compute the weight spaces of $H^0$ (respectively, $H^1$) of the relative tangent bundle of BSDH-varieties
associated to some elements of the Weyl group. 
In Section 6, we prove some results on cohomology modules of the tangent bundle of BSDH varieties.
In Section 7, we prove the  main result using the results from the previous sections.

\section{preliminaries}
In this section, we set up some notation and preliminaries. We refer to \cite{Bri}, \cite{Hum1}, \cite{Hum2}, \cite{Jan} for preliminaries in algebraic groups and Lie algebras.

Let $G$ be a simple algebraic group of adjoint type over $\mathbb{C}$ and $T$ be a maximal torus of
$G$.  Let $W=N_{G}(T)/T$ denote the Weyl group of $G$ with respect to $T$ and we denote 
the set of roots of $G$ with respect to $T$ by $R$. Let $B^{+}$ be a  Borel subgroup of $G$ 
containing $T$. Let $B$ be the Borel subgroup of $G$ opposite to $B^{+}$ determined by $T$. 
That is, $B=n_{0}B^{+}n_{0}^{-1}$, where $n_{0}$ is a representative in $N_{G}(T)$ of the longest     element $w_{0}$ of $W$. Let  $R^{+}\subset R$ be 
the set of positive roots of $G$ with respect to the Borel subgroup $B^{+}$. Note that the set of 
roots of $B$ is equal to the set $R^{-} :=-R^{+}$ of negative roots.

Let $S = \{\alpha_1,\ldots,\alpha_n\}$ denote the set of simple roots in
$R^{+}.$ For $\beta \in R^{+},$ we also use the notation $\beta > 0$.  
The simple reflection in  $W$ corresponding to $\alpha_i$ is denoted
by $s_{\alpha_i}$. Let $\mathfrak{g}$ be the Lie algebra of $G$. 
Let $\mathfrak{h}\subset \mathfrak{g}$ be the Lie algebra of $T$ and  $\mathfrak{b}\subset \mathfrak{g}$ be the Lie algebra of $B$. Let $X(T)$ denote the group of all characters of $T$. 
We have $X(T)\otimes \mathbb{R}=Hom_{\mathbb{R}}(\mathfrak{h}_{\mathbb{R}}, \mathbb{R})$, the dual of the real form of $\mathfrak{h}$. The positive definite 
$W$-invariant form on $Hom_{\mathbb{R}}(\mathfrak{h}_{\mathbb{R}}, \mathbb{R})$ 
induced by the Killing form of $\mathfrak{g}$ is denoted by $(~,~)$. 
We use the notation $\left< ~,~ \right>$ to
denote $\langle \mu, \alpha \rangle  = \frac{2(\mu,
	\alpha)}{(\alpha,\alpha)}$,  for every  $\mu\in X(T)\otimes \mathbb{R}$ and $\alpha\in R$.  
We denote by $X(T)^+$ the set of dominant characters of 
$T$ with respect to $B^{+}$. Let $\rho$ denote the half sum of all 
positive roots of $G$ with respect to $T$ and $B^{+}.$
For any simple root $\alpha$, we denote the fundamental weight
corresponding to $\alpha$  by $\omega_{\alpha}.$  For $1\leq i \leq n,$ let $h(\alpha_{i})\in \mathfrak{h}$ be the fundamental coweight corresponding to $\alpha_{i}.$ That is ; $\alpha_{i}(h(\alpha_{j}))=\delta_{ij},$ where $\delta_{ij}$ is Kronecker delta.

For a simple root $\alpha \in S,$ let $n_{\alpha}\in N_{G}(T)$ be a representative of $s_{\alpha}.$  We denote the minimal parabolic subgroup of $G$ containing $B$ and $n_{\alpha}$ by $P_{\alpha}.$  We recall that the BSDH-variety corresponds to a reduced expression $\underline{i}$ of $w=s_{i_{1}}s_{i_{2}}\cdots s_{i_{r}}$ defined by
$$Z(w,\underline{i})=\frac{P_{\alpha_{1}}\times P_{\alpha_{2}}\times\cdots \times P_{\alpha_{i_{r}}}}{B\times B\times \cdots \times B},$$

where the action of $B\times B\times \cdots \times B$ on $P_{\alpha_{i_{1}}}\times P_{\alpha_{i_{2}}}\times\cdots\times P_{\alpha_{i_{r}}}$ is given by

 $(p_{1}, p_{2}, \dots ,p_{r})(b_{1}, b_{2}\dots, b_{r})=(p_{1}\cdot b_{1}, b_{1}^{-1} \cdot  p_{2}\cdot b_{2},\dots, b_{r-1}^{-1}\cdot p_{r}\cdot b_{r}),$ $p_{j}\in P_{\alpha_{i_{j}}}$, $b_{j}\in B$ and $\underline{i}=(i_{1},i_{2},\dots,i_{r})$ (see \cite[Definition 1, p.73]{De}, \cite[Definition 2.2.1, p.64]{Bri}).

We note that for each reduced expression $\underline{i}$ of $w,$ $Z(w,\underline{i})$ is a smooth projective variety. We denote by $\phi_{w},$ the natural birational surjective morphism from $Z(w, \underline{i})$ to $X(w).$

Let $f_{r}:Z(w, \underline{i})\longrightarrow Z(ws_{i_{r}}, \underline{i'})$ denote the map induced by the projection

 $P_{\alpha_{i_{1}}}\times P_{\alpha_{i_{2}}}\times\cdots\times P_{\alpha_{i_{r}}}\longrightarrow P_{\alpha_{i_{1}}}\times P_{\alpha_{i_{2}}}\times \cdots \times P_{\alpha_{i_{r-1}}},$ where $\underline{i'}=(i_{1},i_{2},\dots, i_{r-1}).$ Then we observe that $f_{r}$ is a $P_{\alpha_{i_{r}}}/B\simeq \mathbb{P}^1$-fibration.

For a $B$-module $V,$ let $\mathcal{L}(w, V)$ denote the restriction of the associated homogeneous vector bundle on $G/B$ to $X(w).$ By abuse of notation, we denote the pull back of $\mathcal{L}(w, V)$ via $\phi_{w}$ to $Z(w, \underline{i})$ also by  $\mathcal{L}(w, V),$ when there is no confusion. Since for any $B$-module $V$ the vector bundle $\mathcal{L}(w, V)$ on $Z(w, \underline{i})$ is the pull back of the homogeneous vector bundle from $X(w),$ we conclude that the cohomology modules

$$H^j(Z(w,\underline{i}), \mathcal{L}(w, V))\simeq H^j(X(w), \mathcal{L}(w, V))$$

for all $j\ge 0$ (see \cite[Theorem 3.3.4(b)]{Bri}), are independent of choice of reduced expression $\underline{i}.$ Hence we denote $H^j(Z(w, \underline{i}), \mathcal{L}(w, V))$ by $H^j(w, V).$ In particular, if $\lambda$ is character of $B,$ then we denote the cohomology modules $H^j(Z(w,\underline{i}), \mathcal{L}_{\lambda})$ by $H^j(w, \lambda).$

We recall the following short exact sequence of $B$-modules from \cite{CKP}, we call it $SES$ :

\begin{enumerate}
	\item [(1)] $H^0(w, V)\simeq H^0(s_{\gamma}, H^0(s_{\gamma}w, V)).$
	
	\item[(2)] $0\rightarrow H^1(s_{\gamma}, H^0(s_{\gamma}w, V))\rightarrow H^1(w, V)\rightarrow  H^0(s_{\gamma}, H^1(s_{\gamma}w, V))\rightarrow0.$
\end{enumerate}

Let $\alpha$ be a simple root and $\lambda\in X(T)$ be such that $\langle \lambda, \alpha\rangle\ge0.$ Let $\mathbb{C}_{\lambda}$ denote one dimensional $B$-module associated to $\lambda.$ Here, we recall the following result due to Demazure \cite[p.271]{Dem} on short exact sequence of $B$-modules:

\begin{lem} \label{lemma 1.1}
	Let $\alpha$ be a simple root and $\lambda\in X(T)$ be such that $\langle \lambda , \alpha \rangle \ge 0.$ Let $ev:H^0(s_{\alpha}, \lambda)\longrightarrow \mathbb{C}_{\lambda}$ be the evaluation map. Then we have 
	\begin{enumerate}
		\item[(1)] If $\langle \lambda, \alpha \rangle =0,$ then $H^0(s_{\alpha}, \lambda)\simeq \mathbb{C}_{\lambda}.$
		\item[(2)] If $\langle \lambda , \alpha \rangle \ge 1,$ then $\mathbb{C}_{s_{\alpha}(\lambda)}\hookrightarrow H^0(s_{\alpha}, \lambda) $, and there is a short exact sequence of $B$-modules:
		$$0\rightarrow H^0(s_{\alpha}, \lambda-\alpha)\longrightarrow H^0(s_{\alpha}, \lambda)/\mathbb{C}_{s_{\alpha}(\lambda)}\longrightarrow \mathbb{C}_{\lambda}\rightarrow 0.$$ Further more, $H^{0}(s_{\alpha}, \lambda- \alpha)=0$ when $\langle\lambda , \alpha \rangle=1.$
		\item[(3)] Let $n=\langle \lambda ,\alpha \rangle.$ As a $B$-module, $H^0(s_{\alpha}, \lambda)$ has a composition series 
		$$0\subseteq V_{n}\subseteq V_{n-1}\subseteq \dots \subseteq V_{0}=H^0(s_{\alpha},\lambda)$$
		
		such that  $V_{i}/V_{i+1}\simeq \mathbb{C}_{\lambda - i\alpha}$ for $i=0,1,\dots,n-1$ and $V_{n}=\mathbb{C}_{s_{\alpha}(\lambda)}.$
		
	\end{enumerate}
\end{lem}

We define the dot action by $w\cdot \lambda= w(\lambda + \rho)-\rho,$ where $\rho$ is the half sum of positive roots. As a consequence of exact sequences of Lemma \ref{lemma 1.1}, we can prove the following.

Let $w\in W$, $\alpha$ be a simple root, and set $v=ws_{\alpha}$.
\begin{lem} \label{lemma 1.2}
	If $l(w) = l(v)+1$, then we have
	\begin{enumerate}
		
		\item  If $\langle \lambda , \alpha \rangle \geq 0$, then 
		$H^{j}(w , \lambda) = H^{j}(v, H^0({s_\alpha, \lambda}) )$ for all $j\geq 0$.
		
		\item  If $\langle \lambda ,\alpha \rangle \geq 0$, then $H^{j}(w , \lambda ) = H^{j+1}(w , s_{\alpha}\cdot \lambda)$ for all $j\geq 0$.
		
		\item If $\langle \lambda , \alpha \rangle  \leq -2$, then $H^{j+1}(w , \lambda ) = H^{j}(w ,s_{\alpha}\cdot \lambda)$ for all $j\geq 0$. 
		
		\item If $\langle \lambda , \alpha \rangle  = -1$, then $H^{j}( w ,\lambda)$ vanishes for every $j\geq 0$.
		
	\end{enumerate}
	
\end{lem}

The following consequence of Lemma \ref {lemma 1.2} will be used to compute the cohomology modules in this paper.
Now onwards we will denote the Levi subgroup of $P_{\alpha}$($\alpha \in S$) containing $T$ by $L_{\alpha}$ and the subgroup $L_{\alpha}\cap B$ by $B_{\alpha}.$ Let $\pi: \tilde{G}\longrightarrow G$ be the universal cover. Let $\tilde{L}_{\alpha}$ (respectively, $\tilde{B_{\alpha}}$) be the inverse image of $L_{\alpha}$ (respectively, $B_{\alpha}$).
\begin{lem}\label{lemma1.3}
	Let $V$ be an irreducible  $L_{\alpha}$-module. Let $\lambda$
	be a character of $B_{\alpha}$. Then we have 
	\begin{enumerate}
		\item As $L_{\alpha}$-modules, $H^j(L_{\alpha}/B_{\alpha}, V \otimes \mathbb C_{\lambda})\simeq V \otimes
		H^j(L_{\alpha}/B_{\alpha}, \mathbb C_{\lambda}).$
		\item If
		$\langle \lambda , \alpha \rangle \geq 0$, then 
		$H^{0}(L_{\alpha}/B_{\alpha} , V\otimes \mathbb{C}_{\lambda})$ 
		is isomorphic as an $L_{\alpha}$-module to the tensor product of $V$ and 
		$H^{0}(L_{\alpha}/B_{\alpha} , \mathbb{C}_{\lambda})$. Further, we have 
		$H^{j}(L_{\alpha}/B_{\alpha} , V\otimes \mathbb{C}_{\lambda}) =0$ 
		for every $j\geq 1$.
		
		\item If
		$\langle \lambda , \alpha \rangle  \leq -2$, then 
		$H^{0}(L_{\alpha}/B_{\alpha} , V\otimes \mathbb{C}_{\lambda})=0$, 
		and $H^{1}(L_{\alpha}/B_{\alpha} , V\otimes \mathbb{C}_{\lambda})$
		is isomorphic to the tensor product of  $V$ and $H^{0}(L_{\alpha}/B_{\alpha} , 
		\mathbb{C}_{s_{\alpha}\cdot\lambda})$. 
		\item If $\langle \lambda , \alpha \rangle  = -1$, then 
		$H^{j}( L_{\alpha}/B_{\alpha} , V\otimes \mathbb{C}_{\lambda}) =0$ 
		for every $j\geq 0$.
	\end{enumerate}
\end{lem}

\begin{proof} Proof (1).
	By \cite[Proposition 4.8, p.53, I]{Jan} and \cite[Proposition 5.12, p.77, I]{Jan}, 
	for all $j\geq 0$, we have the following isomorphism of  $L_{\alpha}$-modules:
	$$H^j(L_{\alpha}/B_{\alpha}, V \otimes \mathbb C_{\lambda})\simeq V \otimes
	H^j(L_{\alpha}/B_{\alpha}, \mathbb C_{\lambda}).$$ 
	
	Proof of (2), (3) and (4) follows from Lemma \ref{lemma 1.2}  by taking $w=s_{\alpha}$ and 
	the fact that $L_{\alpha}/B_{\alpha} \simeq P_{\alpha}/B$.

\end{proof}

Recall the structure of indecomposable 
modules  over   $B_{\alpha}$ and $\widetilde {B}_{\alpha}$ (see \cite[Corollary 9.1, p.130]{BKS}).

\begin{lem}\label{lemma 1.4}
	\begin{enumerate}
		\item
		Any finite dimensional indecomposable $\widetilde{B}_{\alpha}$-module $V$ is isomorphic to 
		$V^{\prime}\otimes \mathbb{C}_{\lambda}$ for some irreducible representation
		$V^{\prime}$ of $\widetilde{L}_{\alpha}$ and for some character $\lambda$ of $\widetilde{B}_{\alpha}$.
		\item
		Any finite dimensional indecomposable $B_{\alpha}$-module $V$ is isomorphic to 
		$V^{\prime}\otimes \mathbb{C}_{\lambda}$ for some irreducible representation
		$V^{\prime}$ of $\widetilde{L}_{\alpha}$ and for some character $\lambda$ of $\widetilde{B}_{\alpha}$.
	\end{enumerate}
\end{lem}

Now onwards we will assume that $G=PSO(2n+1, \mathbb{C})(n\ge 3).$ 
Note that longest element $w_{0}$ of the Weyl group $W$ of $G$ is equal to $-id.$ We recall the following Proposition from \cite[Proposition 1.3, p.858]{YZ}. 

\begin{prop}
	Let $c\in W$ be a Coxeter element, let $\omega_{i}$ be the fundamental weight corresponding to the simple root $\alpha_{i}.$ Then there exists a least positive integer $h(i, c)$ such that
	$c^{h(i,c)}(\omega_{i})=w_{0}(\omega_{i}).$
\end{prop}

\begin{lem}\label{lemma 3.2} Let $c\in W$ be a Coxeter element. Then we have 
	
	\begin{enumerate}
		\item[(1)] $w_{0}=c^n.$ 
		\item[(2)] For any sequence $\underline{i}^{r} (1\le r \le n )$  of reduced expressions of $c;$ the sequence $\underline{i}=(\underline{i}^1,\underline{i}^2,\dots,\underline{i}^r)$ is a reduced expression of $w_{0}.$
	\end{enumerate}
	
\end{lem}
\begin{proof}
	 Note that for $n\ge 3,$ there is an isomorphism of Weyl group of $B_{n}$ and Weyl group of $C_{n}$ sending $s_{i}\mapsto  s_{i}$ for $(1\le i\le n).$
	 Proof of the lemma holds in the case of type $C_{n}$ for ($n\ge 3$) (see \cite[Lemma 4.2, p.441]{CK}). Therefore lemma holds for type $B_{n}$ $(n\ge 3).$
	
\end{proof}

\begin{lem}
	Let $n\ge a_{1}>a_{2}> \dots >a_{r-1}>a_{r}\ge 1$ be a decreasing sequence of integers. Then,
	
	$w=(\prod\limits_{j=a_{1}}^{n} s_{j})(\prod\limits_{j=a_{2}}^{n} s_{j})\cdots (\prod\limits_{j=a_{r-1}}^{n}s_{j}) (\prod\limits_{j=a_{r}}^{n-1} s_{j})$
	is a reduced expression of $w.$
\end{lem}
\begin{proof}
Note that for $n\ge 3,$ there is an isomorphism of Weyl group of $B_{n}$ and Weyl group of $C_{n}$ sending $s_{i}\mapsto  s_{i}$ for $(1\le i\le n).$
Proof of the lemma holds in the case of type $C_{n}$ for ($n\ge 3$)  (see \cite[Lemma 4.3,p.441]{CK}). Therefore lemma holds for type $B_{n}$ $(n\ge 3).$
	
\end{proof}	

Let $c$ be a Coxeter element in $W.$ We take a reduced expression

$c=[a_{1}, n][a_{2}, a_{1}-1]\cdots [a_{k}, a_{k-1}-1],$ where $[i, j]=s_{i}s_{i+1}\cdots s_{j}$ for $i\le j$ and $n\ge a_{1}>a_{2}>\dots>a_{k}=1.$

Then we have following.
\begin{lem} \label{lemma 3.4}
	\begin{enumerate}
		
		\item[(1)]
		
		For all $1\le i \le k-1,$

		$c^i=(\prod\limits_{l_{1}=1}^{i}[a_{l_{1}}, n])(\prod\limits_{l_{2}=i+1}^{k}[a_{l_{2}},  a_{l_{2}-i}-1])(\prod\limits_{l_{3}=1}^{i-1}[a_{l_{k}}, a_{k-i+l_{3}}-1]).$
		
		\item[(2)] For all $k\le j\le n,$
		
		$c^{j}=(\prod\limits_{l_{1}=1}^{k-1}[a_{l_{1}}, n])([a_{k}, n]^{j+1-k})(\prod\limits_{l_{2}=1}^{k-1}[a_{k}, a_{l_{2}}-1]).$
		\item[(3)] The expressions of $c^i$ for $1\le i \le n$ as in $(1)$ and $(2)$ are reduced.
	\end{enumerate}
	
\end{lem}
\begin{proof}
Note that for $n\ge 3,$ there is an isomorphism of Weyl group of $B_{n}$ and Weyl group of $C_{n}$ sending $s_{i}\mapsto  s_{i}$ for $(1\le i\le n).$
Proof of the lemma holds in the case of type $C_{n}$ for ($n\ge 3$)  (see \cite[Lemma 4.4, p.442]{CK}). Therefore lemma holds for type $B_{n}$ $(n\ge 3).$
\end{proof}

 \section{cohomology modules $H^1(w, \alpha_{j})$ where $j\neq n-1$ and $w\in W$ }
 
 In this section, we prove that $H^{1}(w, \alpha_{j})=0$
 for every $w\in W$ and $j\neq n-1.$
 
\begin{lem}\label{lemma 2}
Let $v\in W$  and $\alpha \in S$. Then $H^{1}(s_{j},H^0(v,\alpha))=0$ for $j\neq n.$
 		
 	\end{lem}
 	\begin{proof}
	
	By \cite[Corollary 5.6, p.778]{Ka}  we have $H^{1}(w, \alpha_{n})=0.$  Therefore, we may assume that $\alpha$ is a long simple root. If $H^{1}(s_{j},H^0(v,\alpha))_{\mu}\neq0,$ then there exists an indecomposable $\tilde{L}_{\alpha_{j}}$-summand $V$ of $H^0(v,\alpha)$ such that $H^1(s_{j}, V)_{\mu}\neq 0.$ By Lemma 2.4, we have $V\simeq V'\otimes \mathbb{C}_{\lambda} $ for some character $\lambda$ of $\tilde{B}_{\alpha_{j}}$ and for some irreducible $\tilde{L}_{\alpha_{j}}$-module $V'.$
 		Since $H^1(s_{j}, V)_{\mu}\neq 0$ from Lemma \ref {lemma1.3}(3) we  have $\langle \lambda, \alpha_{j} \rangle\le -2$. Since $\alpha$ is a long root, there exists $w\in W$ such that $w(\alpha)=\alpha_{0}$. Thus $H^0(v, \alpha)\subseteq H^0(vw, \alpha_{0}).$ Again, since $\alpha_{0}$ is highest long root,  $H^0(w_{0}, \alpha_{0})=\mathfrak{g}\longrightarrow H^0(vw, \alpha_{0})$ is surjective. 
Let $\mu'$ be the lowest weight of $V.$ Then by the above argument $\mu'$ is a root. Therefore we have $\mu'=\mu_{1}+\lambda,$ where $\mu_{1}$ is the lowest weight of $V'.$ Hence, we have 
 		$\langle \mu', \alpha_{j} \rangle\le -2$. Since $\alpha_{j}$ is a long root and $\mu'$ is a root, we have $\langle \mu', \alpha_{j} \rangle =-1 , 0, 1.$  This is a contradiction. Thus we have $H^{1}(s_{j},H^0(v,\alpha))_{\mu}=0.$ 
 	\end{proof}

 	\begin{lem}\label{lemma 2.1}
 Let $v\in W$ and $\alpha_{j} \in S$ be such that $j\ne n-1.$ Then we have \\
 	$H^{1}(s_{k},H^0(v,\alpha_{j}))=0,$ for every $k= 1,2,\dots,n.$
 		
 	\end{lem}
 	\begin{proof}
 		
 		\underline {\bfseries{Step 1:}} $H^0(v,\alpha_{j})_{-(\alpha_{n-1} + 2\alpha_{n})}=0.$ \\
 		{\bfseries Case 1:}\\
 		Assume that	$j=n,$ choose an element $u\in W$ of  minimal length such that $u^{-1}(\alpha_{n})=\beta_{0},$ the highest short root. Then we have $H^{0}(v, \alpha_{j})\subseteq H^0(vu, \beta_{0}).$

 		Since $\beta_{0}$ is dominant weight  the natural restriction map
 		$$H^0(w_{0}, \beta_{0})\longrightarrow H^0(vu,\beta_{0})$$
 		is surjective. 
 		
 		Hence $H^0(v, \alpha_{j})_{\mu}\neq0$ implies either $\mu=0$ or $\mu$ is a short root. \\
 		Therefore, we have $H^0(v, \alpha_{j})_{-(\alpha_{n-1} + 2 \alpha_{n})}=0.$
 		\\
 		{\bf Case 2:} \\ Assume that $1\le j \le n-2.$ Note that if $H^0(v,\alpha_{j})_{\mu}\neq 0$ then either $\mu= \alpha_{j}, 0$ or  
 		$\mu\le -\alpha_{j}$ (see \cite[Corollary 4.5, p.678]{CK}).

 		Hence $H^0(v, \alpha_{j})_{-(\alpha_{n-1} + 2\alpha_{n})}=0.$

 		\underline{\bfseries{Step 2:}} If $H^0(v, \alpha_{j})_{-(\beta_{i}+2\alpha_{n})}\ne 0$ for some $1\le i\le n-2,$ then $H^0(v,\alpha_{j})_{-(\beta_{i}+ \alpha_{n})}\ne 0.$ 
 		
 		Proof of Step 2: If $v=id,$ we are done.
 		So choose $1\le t\le n$ such that $l(s_{t}v)=l(v)-1.$
 		Let $v'=s_{t}v.$  Then $H^0(v, \alpha_{j})=H^0(s_{t}, H^0(v', \alpha_{j})).$\\
 		
 		{\bf Case 1:}  Assume that $t=n.$  In this case $\langle -(\beta_{i} + 2 \alpha_{n}), \alpha_{t} \rangle = -2.$
 		If $H^0(v, \alpha_{j})_{-(\beta_{i} + 2\alpha_{n})} \neq 0$, then there is an indecomposable $B_{\alpha_{n}}$-summand $V$ of $H^0(v', \alpha_{j})$ with highest weight $-\beta_{i}.$ Since $\langle -\beta_{i}, \alpha_{n} \rangle= 2,$  we have $H^0(s_{t}, V)_{-(\beta_{i} + \alpha_{n})}\neq 0.$
 		Therefore we have $H^0(v, \alpha_{j})_{-(\beta_{i} + \alpha_{n})}\neq0.$
 		\\
 		{\bf Case 2:} Assume that $t=n-1.$ In this case $\langle -(\beta_{i} + 2 \alpha_{n}), \alpha_{t} \rangle = 1.$ If $H^0(v, \alpha_{j})_{-(\beta_{i} + 2\alpha_{n})} \neq 0$, then there is an indecomposable $B_{\alpha_{n}}$-summand $V$ of  $H^0(v' , \alpha_{j})$ with highest weight $-(\beta_{i} + 2\alpha_{n}).$
 		Thus by induction hypothesis we have $H^0(v', \alpha_{j})_{-(\beta_{i} + \alpha_{n})}\neq 0.$
 		Since $\langle -(\beta_{i} + \alpha_{n}), \alpha_{t}\rangle=0,$ we have 
 		$H^0(v, \alpha_{j})_{-(\beta_{i} + \alpha_{n})}\neq 0.$
 		\\{\bfseries Case 3:} Assume that $1\le t\le n-2.$
 		In this case $\langle -(\beta_{i} + 2 \alpha_{n}) , \alpha_{t} \rangle = -1, 0 $ or $1.$
 		\\
 		
 		Assume that $i=t.$ Then we have $\langle -(\beta_{i} + 2 \alpha_{n}) , \alpha_{t} \rangle = -1.$  If further $H^0(v , \alpha_{j})_{-(\beta_{i} + 2\alpha_{n})}\neq 0$,  then there is an indecomposable $B_{\alpha_{t}}$-summand $V$ of $H^0(v', \alpha_{j})$ 
 		with highest weight $-(\beta_{i+1} + 2\alpha_{n}).$
 		Therefore we have $H^0(v', \alpha_{j})_{-(\beta_{i+1} + 2\alpha_{n})}\neq 0.$ It is clear from Step: 1 that $t+1\le n-2$. Therefore by induction 
 		$H^0(v', \alpha_{j})_{-(\beta_{t+1} + \alpha_{n})}\neq 0.$ 
 		Since $\langle -(\beta_{t+1} + \alpha_{n}) , \alpha_{t} \rangle = 1,$ we have $H^0(v ,\alpha_{j})_{-(\beta_{t} + \alpha_{n})}\neq 0.$ 
 		
 	Assume that $1\le t \le i-2$ or  $ i+1\le t\le n-2.$ Then we have $\langle -(\beta_{i} + 2 \alpha_{n}) , \alpha_{t} \rangle = 0.$ Thus $H^0(v ,\alpha_{j})_{-(\beta_{t} + 2\alpha_{n})} = H^0(v',\alpha_{j})_{-(\beta_{t} + 2\alpha_{n})}\neq 0.$ Therefore by induction 
 		$H^0(v ,\alpha_{j})_{-(\beta_{t} + \alpha_{n})}\neq 0.$ Since $\langle -(\beta_{i} + \alpha_{n}), \alpha_{t}\rangle=0,$ we have $H^0(v ,\alpha_{j})_{-(\beta_{i} + 2\alpha_{n})}\neq 0.$
 		
 Assume that $i=t+1.$ Since $\langle -(\beta_{i} + \alpha_{n}) , \alpha_{t} \rangle = 1,$ then there is an indecomposable $B_{\alpha_{t}}$-summand $V$ of $H^0(v', \alpha_{j})$ 
 such that: $V=\mathbb{C}_{-(\beta_{i}+ 2\alpha_{n})}\oplus \mathbb{C}_{-(\beta_{i-1} + 2\alpha_{n})}$ or $V=\mathbb{C}_{-(\beta_{i} + 2\alpha_{n})}.$
 Then we have $H^0(v', \alpha_{j})_{-(\beta_{t} + 2\alpha_{n})}\neq 0.$  Therefore by induction 
 $H^0(v', \alpha_{j})_{-(\beta_{i} + \alpha_{n})}\neq 0.$ 
 Since $\langle -(\beta_{t} + \alpha_{n}) , \alpha_{t} \rangle = 1$, we have $H^0(v ,\alpha_{j})_{-(\beta_{i} + \alpha_{n})}\neq 0$
 Hence the proof of Step 2.
 		
 		{\bf Proof of Lemma :}
 		\\
 		{\bf Case 1:}   Assume that $k\ne n.$ Then by Lemma \ref {lemma 2} we have $H^1(s_{k}, H^0(v, \alpha_{j}))=0.$

 		{\bf Case 2:}  Assume that $k = n.$
 		By Step 1 we see that $H^0(v, \alpha_{j})_{-(\alpha_{n-1} + 2\alpha_{n})}= 0.$ Note that if $\beta$ is a root such that $H^0(v, \alpha_{j})_{\beta}\neq0$ and $\langle \beta , \alpha_{n}\rangle=-2,$ then we have $\beta=-(\beta_{i} + 2\alpha_{n})$ for some $1\le i\le n-2.$

 		 By Step 2 if $H^0(v,\alpha_{j})_{-(\beta_{i} + 2\alpha_{n})}\neq 0$ for some $1\le i\le n-2$, then $H^0(v,\alpha_{j})_{-(\beta_{i} + \alpha_{n})}\neq 0.$
 		Therefore $\mathbb{C}_{-(\beta_{i}+ \alpha_{n})}\oplus \mathbb{C}_{-(\beta_{i} + 2\alpha_{n})}$ is an indecomposable $B_{\alpha_{n}}$-summand of $H^0(v,\alpha_{j}).$ By Lemma \ref {lemma 1.4}, $\mathbb{C}_{-(\beta_{i}+ \alpha_{n})}\oplus \mathbb{C}_{-(\beta_{i} + 2\alpha_{n})}$ is isomorphic to $V\otimes \mathbb{C}_{-\omega_{n}}$, where $V$ is an irreducible $\tilde{L}_{\alpha_{n}}$-module.
 		Therefore by Lemma \ref {lemma1.3}(4) we have $H^1(s_{k},  \mathbb{C}_{-(\beta_{i}+ \alpha_{n})}\oplus \mathbb{C}_{-(\beta_{i} + 2\alpha_{n})} )=0.$ 
 		Thus our result follows.
 	\end{proof}

\begin{lem}\label{lemma 2.5}
 		Let $w$ be an element of $W$ and $\alpha_{j}$ be an element of $S$ such that $j\neq n-1.$ Then $H^1(w,\alpha_{j})=0.$
\end{lem}
 	\begin{proof}
		We will prove by induction on length of $w$. If length of $w$ is $0$, then $w=id.$ Thus it follows trivially.
 		Now suppose $w\in W$ such that $l(w)\ge 1.$ Then there exists a simple root $\alpha \in S$ such that $l(s_{\alpha}w)=l(w)-1.$
 		Then using SES:
 		\begin{center}
 			$0\longrightarrow H^1(s_{\alpha}, H^0(s_{\alpha}w, \alpha_{j}))\longrightarrow H^1(w, \alpha_{j})\longrightarrow  H^0(s_{\alpha}, H^1(s_{\alpha}w, \alpha_{j}))\longrightarrow 0.$
 		\end{center}
 		From the above SES using induction hypothesis and Lemma \ref{lemma 2.1}, we get $ H^1(w, \alpha_{j}))=0$ for $j\neq n-1.$

 \end{proof}

\section{cohomology module $H^0$ of the relative tangent bundle}
In this section we describe the weights of $H^0$ of the relative tangent bundle.

Notation:

Let $c$ be a Coxeter element of $W.$ We take a reduced expression $c=[a_{1}, n][a_{2}, a_{1}-1]\cdots[a_{k}, a_{k-1}-1],$ where $[i, j]=s_{i}s_{i+1}\cdots s_{j}$ for $i\le j$ and $n\ge a_{1}>a_{2}>\cdots >a_{k}=1.$

Let $\beta_{i} =\alpha_{i}+\alpha_{i+1}+\cdots+\alpha_{n-1}$ for all $1\le i\le n-1.$

For $1\le r\le k,$ let $n\ge a_{1}>a_{2}>a_{3}>\dots>a_{r}\ge 1$ be a decreasing sequence of integers.

Let

$w_{r}=(\prod\limits_{j=a_{1}}^{n}s_{j})(\prod\limits_{j=a_{2}}^{n}s_{j})(\prod\limits_{j=a_{3}}^{n}s_{j})\cdots(\prod\limits_{j=a_{r-1}}^{n}s_{j})(\prod\limits_{j=a_{r}}^{n-1}s_{j})$

and let

$\tau_{r}=(\prod\limits_{j=a_{1}}^{n}s_{j})(\prod\limits_{j=a_{2}}^{n}s_{j})(\prod\limits_{j=a_{3}}^{n}s_{j})\cdots(\prod\limits_{j=a_{r-1}}^{n}s_{j})(\prod\limits_{j=a_{r}}^{n-2}s_{j}).$

Note that $l(w_{r})=l(\tau_{r}) + 1.$

\begin{lem}\label{lemma 5.1}
Assume that $r\ge 3.$

\begin{enumerate}
\item[(1)] Let $v= s_{a_{r-1}}s_{a_{r-1}+1}\cdots s_{n}s_{a_{r}}s_{a_{r}+1}\cdots s_{n-1}.$ Then we have 
	
	$H^0(v,\alpha_{n-1})=\bigoplus\limits_{i=a_{r}}^{a_{r-1}-1}(\mathbb{C}_{-(\beta_{i} + \alpha_{n})} \oplus \mathbb{C}_{-(\beta_{i} + 2\alpha_{n})}\oplus\mathbb{C}_{-(\beta_{i} + 2\alpha_{n} + \beta_{n-1})} \oplus \cdots \oplus\mathbb{C}_{-(\beta_{i} + 2\alpha_{n} + \beta_{a_{r-1}})})$

	$\bigoplus\limits_{i=a_{r-1}}^{n-2} 
	(\mathbb{C}_{-(\beta_{i} + 2\alpha_{n})}\oplus \mathbb{C}_{-(\beta_{i} + 2\alpha_{n}+\beta_{n-1})}\oplus \cdots \oplus \mathbb{C}_{-(\beta_{i} + 2\alpha_{n} + \beta_{i+1})})\oplus \mathbb{C}_{-(\beta_{n-1} + 2\alpha_{n})}
	.$
	
\item[(2)] Let $v'=s_{1}\cdots s_{n}s_{1}\cdots s_{n-1}.$ Then we have

$H^0(v', \alpha_{n-1})=\bigoplus\limits_{i=1}^{n-2} 
(\mathbb{C}_{-(\beta_{i} + 2\alpha_{n})}\oplus \mathbb{C}_{-(\beta_{i} + 2\alpha_{n}+\beta_{n-1})}\oplus \cdots \oplus \mathbb{C}_{-(\beta_{i} + 2\alpha_{n} + \beta_{i+1})})\oplus \mathbb{C}_{-(\beta_{n-1} + 2\alpha_{n})}
.$

\end{enumerate}	

\end{lem}

\begin{proof}

Proof of (1): Let $u =s_{a_{r}}s_{a_{r}+1}\cdots s_{n-1}.$ By
using SES we have

 $H^0(u, \alpha_{n-1})=\mathbb{C}h(\alpha_{n-1})\oplus(\bigoplus\limits_{j=a_{r}}^{n-1}\mathbb{C}_{-\beta_{j}}).$
 Since $\langle -\beta_{j} , \alpha_{n} \rangle=2$  for all $a_{r}\le j\le n-1,$ 

 by using SES we have
 
 $H^0(s_{n}u, \alpha_{n-1})= \mathbb{C}h({\alpha_{n-1}})\oplus (\bigoplus\limits_{j=a_{r}}^{n-1}\mathbb{C}_{-\beta_{j}}\oplus \mathbb{C}_{-(\beta_{j} + \alpha_{n})}\oplus \mathbb{C}_{-(\beta_{j} + 2\alpha_{n})}).$       \hspace{3.5cm} (4.1.1)
 
Since $\mathbb{C}h({\alpha_{n-1}}) \oplus \mathbb{C}_{-\alpha_{n-1}}$ is an indecomposable $B_{\alpha_{n-1}}$-module (see \cite[p.11]{CKP} and \cite[p.8]{Ka}), by Lemma \ref{lemma 1.4}, we have 

$\mathbb{C}h(\alpha_{n-1}) \oplus \mathbb{C}_{-\alpha_{n-1}}=V\otimes \mathbb{C}_{-\omega_{n-1}},$ 
 where $V$ is the two dimensional irreducible representation of $\tilde{L}_{\alpha_{n-1}}.$

Therefore by Lemma \ref{lemma1.3}(4) we have

 $H^0(\tilde{L}_{\alpha_{n-1}}/\tilde{B}_{\alpha_{n-1}}, \mathbb{C}h({\alpha_{n-1}}) \oplus \mathbb{C}_{-\alpha_{n-1}} )=0.$

Moreover, we have $\langle -(\beta_{n-1} + \alpha_{n}), \beta_{n-1}\rangle=-1$,  $\langle -(\beta_{n-1} + 2\alpha_{n}), \beta_{n-1}\rangle=0,$  $\langle - \beta_{j} , \beta_{n-1} \rangle = -1$, $\langle - (\beta_{j} +\alpha_{n}), \beta_{n-1} \rangle = 0$  and $\langle - (\beta_{j} + 2\alpha_{n}) , \beta_{n-1} \rangle = 1$ for all $a_{r}\le j \le n-2.$  

Therefore we have 

$H^0(s_{n-1}s_{n}u, \alpha_{n-1})= (\bigoplus\limits_{j=a_{r}}^{n-2}\mathbb{C}_{-(\beta_{j}+\alpha_{n})} \oplus \mathbb{C}_{-(\beta_{j} + 2\alpha_{n})}\oplus \mathbb{C}_{-(\beta_{j} + 2\alpha_{n} + \beta_{n-1})})\oplus \mathbb{C}_{-(\beta_{n-1} + 2\alpha_{n})}.$  \hspace{.3cm} (4.1.2)
 
{\bf Claim:} For $ a_{r-1}\le k\le n-2$

$H^0(s_{k}s_{k+1}\cdots s_{n}u, \alpha_{n-1})=\bigoplus\limits_{j=a_{r}}^{k-1}(\mathbb{C}_{-(\beta_{j} + \alpha_{n})} \oplus \mathbb{C}_{-(\beta_{j} + 2\alpha_{n})}\oplus\mathbb{C}_{-(\beta_{j} + 2\alpha_{n} + \beta_{n-1})} \oplus \cdots \oplus\mathbb{C}_{-(\beta_{j} + 2\alpha_{n} + \beta_{k})})$
$\bigoplus\limits_{j=k}^{n-2} 
(\mathbb{C}_{-(\beta_{j} + 2\alpha_{n})}\oplus \mathbb{C}_{-(\beta_{j} + 2\alpha_{n}+\beta_{n-1})}\oplus \cdots \oplus \mathbb{C}_{-(\beta_{j} + 2\alpha_{n} + \beta_{j+1})})\oplus \mathbb{C}_{-(\beta_{n-1} + 2\alpha_{n})}.$

Proof of the claim:
We will prove by descending induction on $k.$ 

By hypothesis we have 

$H^0(s_{k+1}\cdots s_{n}u, \alpha_{n-1})=\bigoplus\limits_{j=a_{r}}^{k}(\mathbb{C}_{-(\beta_{j} + \alpha_{n})} \oplus \mathbb{C}_{-(\beta_{j} + 2\alpha_{n})}\oplus\mathbb{C}_{-(\beta_{j} + 2\alpha_{n} + \beta_{n-1})} \oplus \cdots \oplus\mathbb{C}_{-(\beta_{j} + 2\alpha_{n} + \beta_{k+1})})$
$\bigoplus\limits_{j=k+1}^{n-2} 
(\mathbb{C}_{-(\beta_{j} + 2\alpha_{n})}\oplus \mathbb{C}_{-(\beta_{j} + 2\alpha_{n}+\beta_{n-1})}\oplus \cdots \oplus \mathbb{C}_{-(\beta_{j} + 2\alpha_{n} + \beta_{j+1})})\oplus \mathbb{C}_{-(\beta_{n-1} + 2\alpha_{n})}. \hspace{2.9cm} (4.1.3)$

Let $V=\bigoplus\limits_{j=a_{r}}^{k-1}(\mathbb{C}_{-(\beta_{j} + \alpha_{n})} \oplus \mathbb{C}_{-(\beta_{j} + 2\alpha_{n})}\oplus\mathbb{C}_{-(\beta_{j} + 2\alpha_{n} + \beta_{n-1})} \oplus \cdots \oplus\mathbb{C}_{-(\beta_{j} + 2\alpha_{n} + \beta_{k+1})})$ and $V'=\bigoplus\limits_{j=k+2}^{n-2} 
(\mathbb{C}_{-(\beta_{j} + 2\alpha_{n})}\oplus \mathbb{C}_{-(\beta_{j} + 2\alpha_{n}+\beta_{n-1})}\oplus \cdots \oplus \mathbb{C}_{-(\beta_{j} + 2\alpha_{n} + \beta_{j+1})}) \oplus \mathbb{C}_{-(\beta_{n-1} + 2\alpha_{n})}.$

Then roots $\{(\beta_{j} + \alpha_{n}), (\beta_{j} + 2\alpha_{n}), (\beta_{j} + 2\alpha_{n} + \beta_{n-1}),\dots,(\beta_{j} + 2\alpha_{n} + \beta_{k+2}): a_{r}\le j \le k-1\},$ $\{(\beta_{j} + 2\alpha_{n}), (\beta_{j} + 2\alpha_{n} + \beta_{n-1}),\dots,(\beta_{j} + 2\alpha_{n} + \beta_{j+1}): k+2\le j \le n-2\}$ and $-(\beta_{n-1} + 2\alpha_{n})$
are orthogonal to $\alpha_{k}.$

Therefore $V$, $V'$ are direct sums of irreducible $\tilde{L}_{\alpha_{k}}-$modules. By Lemma \ref{lemma1.3}(2), we have 

 $H^0(\tilde{L}_{\alpha_{k}}/\tilde{B}_{\alpha_{k}}, V)= V,$
$H^0(\tilde{L}_{\alpha_{k}}/\tilde{B}_{\alpha_{k}}, V')= V'$  and

 $H^0(\tilde{L}_{\alpha_{k}}/\tilde{B}_{\alpha_{k}}, \mathbb{C}_{-(\beta_{n-1} + 2\alpha_{n})})=\mathbb{C}_{-(\beta_{n-1} + 2\alpha_{n})}.$

Further the remaining roots of $(4.1.3)$ are $\{-(\beta_{k} + \alpha_{n}), -(\beta_{k} + 2\alpha_{n}), -(\beta_{k} +2\alpha_{n} + \beta_{n-1}),\ldots, -(\beta_{k}+2\alpha_{n}+ \beta_{k+2}), -(\beta_{k} + 2\alpha_{n}+ \beta_{k+1})\},$ $\{-(\beta_{k+1} + 2\alpha_{n}),-(\beta_{k+1}+2\alpha_{n}+\beta_{n-1}), \ldots , -(\beta_{k} + 2\alpha_{n}+ \beta_{k+2})\}$ and $\{-(\beta_{j} + 2\alpha_{n}+ \beta_{k+1}): a_{r}\le j \le k\}.$

 Since $\langle -(\beta_{k}+\alpha_{n}), \alpha_{k}\rangle=-1,$  by Lemma \ref{lemma1.3}(4) we have 
 
 $ H^0(\tilde{L}_{\alpha_{k}}/\tilde{B}_{\alpha_{k}}, \mathbb{C}_{-(\beta_{k} + \alpha_{n})})=0.$ 
 
Since $\mathbb{C}_{-(\beta_{k} + 2\alpha_{n})}\oplus \mathbb{C}_{-(\beta_{k+1} + 2\alpha_{n})}$  is the irreducible two dimensional $\tilde{L}_{\alpha_{k}}$-module, by Lemma \ref {lemma1.3}(2) we have
 
 $H^0(\tilde{L}_{\alpha_{k}}/\tilde{B}_{\alpha_{k}}, \mathbb{C}_{-(\beta_{k} + 2\alpha_{n})}\oplus \mathbb{C}_{-(\beta_{k+1} + 2\alpha_{n})} )= \mathbb{C}_{-(\beta_{k} + 2\alpha_{n})}\oplus \mathbb{C}_{-(\beta_{k+1} + 2\alpha_{n})}.$

Similarly for each $k+2\le j\le n-1,$  $\mathbb{C}_{-(\beta_{k} + 2\alpha_{n}+ \beta_{j})}\oplus \mathbb{C}_{-(\beta_{k+1} + 2\alpha_{n}+ \beta_{j})}$  is the irreducible two dimensional $\tilde{L}_{\alpha_{k}}-$module. Therefore by Lemma \ref {lemma1.3}(2) we have

$H^0(\tilde{L}_{\alpha_{k}}/\tilde{B}_{\alpha_{k}}, \mathbb{C}_{-(\beta_{k} + 2\alpha_{n}+ \beta_{j})} \oplus \mathbb{C}_{-(\beta_{k+1} + 2\alpha_{n} + \beta_{j})})= \mathbb{C}_{-(\beta_{k} + 2\alpha_{n} + \beta_{j})} \oplus \mathbb{C}_{-(\beta_{k+1} + 2\alpha_{n} + \beta_{j})})$ for each $k+2\le j\le n-1.$

Moreover, $\langle-(\beta_{j} + 2\alpha_{n} + \beta_{k+1}) , \alpha_{k}\rangle=1$ for all $a_{r}\le j \le k-1 $. Therefore by Lemma \ref {lemma1.3}(2) we have  

$H^0(\tilde{L}_{\alpha_{k}}/\tilde{B}_{\alpha_{k}}, \mathbb{C}_{-(\beta_{j} + 2\alpha_{n}+ \beta_{k+1})})=\mathbb{C}_{-(\beta_{j} +
2\alpha_{n} + \beta_{k+1})} \oplus \mathbb{C}_{-(\beta_{j} + 2\alpha_{n}+ \beta_{k})}$ for all   $a_{r}\le j \le k-1 .$ 

Since $\langle-(\beta_{k} + 2\alpha_{n} + \beta_{k+1}) , \alpha_{k}\rangle=0,$ by Lemma \ref {lemma1.3}(2)  we have

 $H^0(\tilde{L}_{\alpha_{k}}/\tilde{B}_{\alpha_{k}}, \mathbb{C}_{-(\beta_{k} + 2\alpha_{n}+ \beta_{k+1})})=\mathbb{C}_{-(\beta_{k} +
	2\alpha_{n}+ \beta_{k+1})}.$

From the above discussion, we have

$H^0(s_{k}s_{k+1}\cdots s_{n}u,
\alpha_{n-1})=\bigoplus\limits_{j=a_{r}}^{k-1}(\mathbb{C}_{-(\beta_{j} + \alpha_{n})} \oplus \mathbb{C}_{-(\beta_{j} + 2\alpha_{n})} \oplus \mathbb{C}_{-(\beta_{j} + 2\alpha_{n} + \beta_{n-1})} \oplus \cdots \oplus \mathbb{C}_{-(\beta_{j} + 2\alpha_{n} + \beta_{k})})$ 
$\bigoplus \limits_{j=k}^{n-2} 
(\mathbb{C}_{-(\beta_{j} + 2\alpha_{n})} \oplus \mathbb{C}_{-(\beta_{j} + 2\alpha_{n}+\beta_{n-1})} \oplus \cdots \oplus \mathbb{C}_{-(\beta_{j} + 2\alpha_{n} + \beta_{j+1})}) \oplus \mathbb{C}_{-(\beta_{n-1} + 2\alpha_{n})}.$

Therefore the claim follows.

The claim implies

 $H^0(v, \alpha_{n-1})=H^0(s_{a_{r-1}}s_{a_{r}+1}\cdots s_{n}v'_{r},\alpha_{n-1})=\bigoplus\limits_{j=a_{r}}^{a_{r-1}-1}(\mathbb{C}_{-(\beta_{j} + \alpha_{n})} \oplus \mathbb{C}_{-(\beta_{j} + 2\alpha_{n})} \oplus \mathbb{C}_{-(\beta_{j} + 2\alpha_{n} + \beta_{n-1})} \oplus \cdots \oplus \mathbb{C}_{-(\beta_{j} + 2\alpha_{n} + \beta_{a_{r-1}})})$
 $\bigoplus\limits_{j=a_{r-1}}^{n-2} 
 (\mathbb{C}_{-(\beta_{j} + 2\alpha_{n})} \oplus \mathbb{C}_{-(\beta_{j} + 2\alpha_{n}+ \beta_{n-1})} \oplus \cdots \oplus \mathbb{C}_{-(\beta_{j} + 2\alpha_{n} + \beta_{j+1})})\oplus \mathbb{C}_{-(\beta_{n-1} + 2\alpha_{n})}
 .$

Proof of (2): By (1) we have

$H^0(s_{2}s_{3}\cdots s_{n}s_{1}s_{2}\cdots s_{n-1}, \alpha_{n-1} )=(\mathbb{C}_{-(\beta_{1} + \alpha_{n})} \oplus \mathbb{C}_{-(\beta_{1} + 2\alpha_{n})} \oplus \mathbb{C}_{-(\beta_{1} + 2\alpha_{n} + \beta_{n-1})} \oplus \cdots \oplus \mathbb{C}_{-(\beta_{1} + 2\alpha_{n} + \beta_{2})})$
$\bigoplus\limits_{j=2}^{n-2} 
(\mathbb{C}_{-(\beta_{j} + 2\alpha_{n})}\oplus \mathbb{C}_{-(\beta_{j} + 2\alpha_{n}+ \beta_{n-1})} \oplus \cdots \oplus \mathbb{C}_{-(\beta_{j} + 2\alpha_{n} + \beta_{j+1})})\oplus \mathbb{C}_{-(\beta_{n-1} + 2\alpha_{n})}.
\hspace{.5cm}(4.1.4)$

 The roots of (4.1.4)  are $\{-(\beta_{1} + \alpha_{n}), -(\beta_{1} + 2\alpha_{n}), -(\beta_{1} +2\alpha_{n} + \beta_{n-1}),\ldots, -(\beta_{1}+2\alpha_{n}+ \beta_{3}), -(\beta_{1} + 2\alpha_{n}+ \beta_{2})\},$ $\{-(\beta_{2} + 2\alpha_{n}),-(\beta_{2}+2\alpha_{n}+\beta_{n-1}), \ldots , -(\beta_{1} + 2\alpha_{n}+ \beta_{3})\}$ and $-(\beta_{1} + 2\alpha_{n}+ \beta_{2}).$

 Since $-(\beta_{n-1} + 2\alpha_{n})$
 is orthogonal to $\alpha_{1},$ by Lemma \ref{lemma1.3}(2) we have

  $H^0(\tilde{L}_{\alpha_{1}}/\tilde{B}_{\alpha_{1}}, \mathbb{C}_{-(\beta_{n-1} + \alpha_{n})})=\mathbb{C}_{-(\beta_{n-1} + \alpha_{n})}.$ 
 
Since $\langle -(\beta_{1} + \alpha_{n}) , \alpha_{1} \rangle=-1,$  by Lemma \ref{lemma1.3}(4) we have 

$ H^0(\tilde{L}_{\alpha_{1}}/\tilde{B}_{\alpha_{1}}, \mathbb{C}_{-(\beta_{1} + \alpha_{n})})=0.$ 

Since $\mathbb{C}_{-(\beta_{1} + 2\alpha_{n})}\oplus \mathbb{C}_{-(\beta_{2} + 2\alpha_{n})}$  is the irreducible two dimensional $\tilde{L}_{\alpha_{1}}$-module, by Lemma \ref {lemma1.3}(2) we have

$H^0(\tilde{L}_{\alpha_{1}}/\tilde{B}_{\alpha_{1}}, \mathbb{C}_{-(\beta_{1} + 2\alpha_{n})}\oplus \mathbb{C}_{-(\beta_{2} + 2\alpha_{n})} )= \mathbb{C}_{-(\beta_{1} + 2\alpha_{n})} \oplus \mathbb{C}_{-(\beta_{2} + 2\alpha_{n})}.$

Similarly for each $3\le j\le n-1,$  $\mathbb{C}_{-(\beta_{1} + 2\alpha_{n} + \beta_{j})} \oplus \mathbb{C}_{-(\beta_{2} + 2\alpha_{n} + \beta_{j})}$  is the irreducible two dimensional $\tilde{L}_{\alpha_{1}}-$module. Therefore by Lemma \ref {lemma1.3}(2) we have

$H^0(\tilde{L}_{\alpha_{1}}/\tilde{B}_{\alpha_{1}}, \mathbb{C}_{-(\beta_{1} + 2\alpha_{n} + \beta_{j})} \oplus \mathbb{C}_{-(\beta_{2} + 2\alpha_{n} + \beta_{j})})= \mathbb{C}_{-(\beta_{1} + 2\alpha_{n}+ \beta_{j})} \oplus \mathbb{C}_{-(\beta_{2} + 2\alpha_{n}+ \beta_{j})})$ for each $3\le j\le n-1.$

Since $\langle-(\beta_{1} + 2\alpha_{n} + \beta_{2}) , \alpha_{1}\rangle=0,$ by Lemma \ref {lemma1.3}(2)  we have

$H^0(\tilde{L}_{\alpha_{1}}/\tilde{B}_{\alpha_{1}}, \mathbb{C}_{-(\beta_{1} + 2\alpha_{n} + \beta_{2})})=\mathbb{C}_{-(\beta_{1} +
	2\alpha_{n}+ \beta_{2})}.$




 From the above discussion, we have

$H^0(v', \alpha_{n-1})=\bigoplus\limits_{i=1}^{n-2} 
(\mathbb{C}_{-(\beta_{i} + 2\alpha_{n})}\oplus \mathbb{C}_{-(\beta_{i} + 2\alpha_{n} + \beta_{n-1})} \oplus \cdots \oplus \mathbb{C}_{-(\beta_{i} + 2\alpha_{n} + \beta_{i+1})})\oplus \mathbb{C}_{-(\beta_{n-1} + 2\alpha_{n})}
.$
\end{proof}

\begin{lem}\label {lemma 5.2}
Let $3\le r \le k$ and let $v= s_{a_{r-1}}s_{a_{r-1}+1}\cdots s_{n}s_{a_{r}}s_{a_{r}+1}\cdots s_{n-1}.$ Then we have
 
\begin{enumerate}
 	\item 
 
 $H^0(s_{a_{r-2}}\cdots s_{n}v, \alpha_{n-1})=\bigoplus\limits_{i=a_{r}}^{a_{r-1}-1}(\mathbb{C}_{-(\beta_{i} + 2\alpha_{n} + \beta_{a_{r-2}-1})} \oplus \cdots \oplus\mathbb{C}_{-(\beta_{i} + 2\alpha_{n} + \beta_{a_{r-1}})})$
 
 $\bigoplus\limits_{i=a_{r-1}}^{a_{r-2}-2} 
 (\mathbb{C}_{-(\beta_{i} + 2\alpha_{n} + \beta_{a_{r-2}-1})}\oplus \cdots \oplus \mathbb{C}_{-(\beta_{i} + 2\alpha_{n} + \beta_{i+1})})
 .$
 
\item[(2)]$H^0(w_{r}, \alpha_{n-1})=\bigoplus\limits_{i=a_{r}}^{a_{r-1}-1}(\mathbb{C}_{-(\beta_{i} + 2\alpha_{n} + \beta_{a_{r-2}-1})} \oplus \cdots \oplus \mathbb{C}_{-(\beta_{i} + 2\alpha_{n} + \beta_{a_{r-1}})})$

$\bigoplus\limits_{i=a_{r-1}}^{a_{r-2}-2} 
(\mathbb{C}_{-(\beta_{i} + 2\alpha_{n} + \beta_{a_{r-2}-1})}\oplus \cdots \oplus \mathbb{C}_{-(\beta_{i} + 2\alpha_{n} + \beta_{i+1})}).$

\item[(3)] Let $u_{1}=(s_{a_{1}}\cdots s_{n})(s_{a_{2}}\cdots s_{n})\cdots (s_{a_{k-1}}\cdots s_{n})v',$ where $v'$ is defined as in Lemma \ref {lemma 5.1}. Then we have

$H^0(u_{1}, \alpha_{n-1})=\bigoplus\limits_{i=1}^{a_{k-1}-2} 
(\mathbb{C}_{-(\beta_{i} + 2\alpha_{n} + \beta_{a_{k-1}-1})} \oplus \cdots \oplus \mathbb{C}_{-(\beta_{i} + 2\alpha_{n} + \beta_{i+1})}).$

\end{enumerate}
\end{lem}
\begin{proof}
Proof of (1):
Since $r\ge 3,$ we have $a_{r-1}<n.$ By Lemma \ref {lemma 5.1}(1), we have 

	$H^0(v,\alpha_{n-1})=\bigoplus\limits_{i=a_{r}}^{a_{r-1}-1}(\mathbb{C}_{-(\beta_{i} + \alpha_{n})} \oplus \mathbb{C}_{-(\beta_{i} + 2\alpha_{n})} \oplus \mathbb{C}_{-(\beta_{i} + 2\alpha_{n} + \beta_{n-1})} \oplus \cdots \oplus \mathbb{C}_{-(\beta_{i} + 2\alpha_{n} + \beta_{a_{r-1}})})$
$\bigoplus\limits_{i=a_{r-1}}^{n-2} 
(\mathbb{C}_{-(\beta_{i} + 2\alpha_{n})}\oplus \mathbb{C}_{-(\beta_{i} + 2\alpha_{n}+\beta_{n-1})} \oplus \cdots \oplus \mathbb{C}_{-(\beta_{i} + 2\alpha_{n} + \beta_{i+1})}) \oplus \mathbb{C}_{-(\beta_{n-1} + 2\alpha_{n})}
.$

Since $\{-(\beta_{i} + 2\alpha_{n}+\beta_{n-1}), \ldots, -(\beta_{i} + 2\alpha_{n} + \beta_{a_{r-1}}): a_{r}\le i \le a_{r-1}-1\}$ are orthogonal to $\alpha_{n}$, by Lemma \ref {lemma1.3}(2), we have 

$H^0(\tilde{L}_{\alpha_{n}}/\tilde{B}_{\alpha_{n}}, \mathbb{C}_{-(\beta_{i} + 2\alpha_{n} + \beta_{t})})=\mathbb{C}_{-(\beta_{i} +
	2\alpha_{n}+ \beta_{t})}$ for all $a_{r}\le i\le a_{r-1}-1$ and $a_{r-1} \le t \le n-1.$

Since $\{-(\beta_{i} + 2\alpha_{n} + \beta_{n-1}), \ldots, -(\beta_{i} + 2\alpha_{n} + \beta_{i+1}): a_{r-1}\le i \le n-2\}$ are orthogonal to $\alpha_{n}$, by Lemma \ref {lemma1.3}(2) we have 

$H^0(\tilde{L}_{\alpha_{n}}/\tilde{B}_{\alpha_{n}}, \mathbb{C}_{-(\beta_{i} + 2\alpha_{n} + \beta_{l})})=\mathbb{C}_{-(\beta_{i} +
	2\alpha_{n} + \beta_{l})}$ for all  $ i+1 \le l \le n-1,$ where $a_{r-1}\le i\le n-2.$

Since 
$\langle -(\beta_{i}+ 2\alpha_{n}),\alpha_{n}\rangle=-2$ for all $a_{r}\le i\le n-1,$ by Lemma \ref {lemma1.3}(3) we have 

$H^0(\tilde{L}_{\alpha_{n}}/\tilde{B}_{\alpha_{n}}, \mathbb{C}_{-(\beta_{n-1} + 2\alpha_{n})})=0$ for all $a_{r-1}\le i\le n-1.$

Moreover, for each $a_{r}\le i\le a_{r-1}-1,$ $\mathbb{C}_{-(\beta_{i} + \alpha_{n})} \oplus \mathbb{C}_{-(\beta_{i} + 2\alpha_{n})}$ is an indecomposable two dimensioal $B_{\alpha_{n}}$-module. Therefore by Lemma \ref {lemma 1.4}, we have $\mathbb{C}_{-(\beta_{i} + \alpha_{n})} \oplus \mathbb{C}_{-(\beta_{i} + 2\alpha_{n})}=V_{i}\otimes \mathbb{C}_{-\omega_{n}},$
where $V_{i}$ is the irreducible two dimensional representation of $\tilde{L}_{\alpha_{n}}.$
By Lemma \ref {lemma1.3}(4) we have

$H^0(\tilde{L}_{\alpha_{n}}/\tilde{B}_{\alpha_{n}}, \mathbb{C}_{-(\beta_{i} + \alpha_{n})} \oplus \mathbb{C}_{-(\beta_{i} + 2\alpha_{n})})=0$  for each $a_{r}\le i\le a_{r-1}-1.$


From the above discussion, we have 

$H^0(s_{n}v_{r}, \alpha_{n-1})= \bigoplus\limits_{i=a_{r}}^{a_{r-1}-1}(\mathbb{C}_{-(\beta_{i} + 2\alpha_{n} + \beta_{n-1})} \oplus \cdots \oplus \mathbb{C}_{-(\beta_{i} + 2\alpha_{n} + \beta_{a_{r-1}})})$
$\bigoplus\limits_{i=a_{r-1}}^{n-2} 
(\mathbb{C}_{-(\beta_{i} + 2\alpha_{n} + \beta_{n-1})} \oplus \cdots \oplus \mathbb{C}_{-(\beta_{i} + 2\alpha_{n} + \beta_{i+1})}).$

Since $\langle -(\beta_{i} + 2\alpha_{n} + \beta_{n-1}) , \alpha_{n-1} \rangle= -1$  for each $a_{r}\le i \le n-2,$ by Lemma \ref{lemma1.3}(4) we have  $H^0(\tilde{L}_{\alpha_{n-1}}/\tilde{B}_{\alpha_{n-1}}, \mathbb{C}_{-(\beta_{i} + \alpha_{n} + \beta_{n-1})})=0$  for each $a_{r}\le i\le n-2.$

Moreover, $\{-(\beta_{i} + 2\alpha_{n} + \beta_{n-2}), \dots, -(\beta_{i} + 2\alpha_{n}+ \beta_{a_{r-1}}): a_{r}\le i\le a_{r-1}-1\}$, $\{-(\beta_{i} + 2\alpha_{n}+ \beta_{n-2}), \dots, -(\beta_{i} + 2\alpha_{n}+ \beta_{i+1}): a_{r-1}\le i\le n-3\}$, and $\beta_{n-1}+2\alpha_{n}$  are orthogonal to $\alpha_{n-1}.$

Therefore  we have 

$H^0(s_{n-1}s_{n}v, \alpha_{n-1})= \bigoplus\limits_{i=a_{r}}^{a_{r-1}-1}(\mathbb{C}_{-(\beta_{i} + 2\alpha_{n} + \beta_{n-2})} \oplus \cdots \oplus \mathbb{C}_{-(\beta_{i} + 2\alpha_{n} + \beta_{a_{r-1}})})$
$\bigoplus\limits_{i=a_{r-1}}^{n-3} 
(\mathbb{C}_{-(\beta_{i} + 2\alpha_{n} + \beta_{n-2})} \oplus \cdots \oplus \mathbb{C}_{-(\beta_{i} + 2\alpha_{n} + \beta_{i+1})})
.$

Proceeding recursively, we have 

$H^0(s_{a_{r-2}}s_{a_{r-2}+1}\cdots s_{n}v, \alpha_{n-1})= \bigoplus\limits_{i=a_{r}}^{a_{r-1}-1}(\mathbb{C}_{-(\beta_{i} + 2\alpha_{n} + \beta_{a_{r-2}-1})} \oplus \cdots \oplus \mathbb{C}_{-(\beta_{i} + 2\alpha_{n} + \beta_{a_{r-1}})})\bigoplus$ \\
$\bigoplus\limits_{i=a_{r-1}}^{a_{r-2}-2} 
(\mathbb{C}_{-(\beta_{i} + 2\alpha_{n} + \beta_{a_{r-2}-1})} \oplus \cdots \oplus \mathbb{C}_{-(\beta_{i} + 2\alpha_{n} + \beta_{i+1})}).$

Proof of (2):

Since $\{-(\beta_{i} + 2\alpha_{n} + \beta_{a_{r-2}-1}), \ldots , -(\beta_{i} + 2\alpha_{n} + \beta_{a_{r-1}}): a_{r}\le i\le a_{r-1}-1\},$ $\{-(\beta_{i} +2\alpha_{n} + \beta_{a_{r-2}-1}), \ldots , -(\beta_{i} + 2\alpha_{n} + \beta_{i+1}): a_{r-1}\le i\le a_{r-2}-2\}$ are orthogonal to 
$\alpha_{j}$ for all $a_{r-3}\le j\le n,$ by Lemma \ref{lemma1.3}(2) we have 

$H^0(s_{a_{r-3}}\cdots s_{n}s_{a_{r-2}}\cdots s_{n}v, \alpha_{n-1})=H^0(s_{a_{r-2}}\cdots s_{n}v, \alpha_{n-1}).$

Proceeding recursively we have 

$H^0(w_{r}, \alpha_{n-1})=H^0(s_{a_{r-2}}\cdots s_{n}v, \alpha_{n-1})=\bigoplus\limits_{i=a_{r}}^{a_{r-1}-1}(\mathbb{C}_{-(\beta_{i} + 2\alpha_{n} + \beta_{a_{r-2}-1})} \oplus \cdots \oplus \mathbb{C}_{-(\beta_{i} + 2\alpha_{n} + \beta_{a_{r-1}})})$
$\bigoplus\limits_{i=a_{r-1}}^{a_{r-2}-2} 
(\mathbb{C}_{-(\beta_{i} + 2\alpha_{n} + \beta_{a_{r-2}-1})} \oplus \cdots \oplus \mathbb{C}_{-(\beta_{i} + 2\alpha_{n} + \beta_{i+1})}).$

Proof of (3):

By Lemma \ref {lemma 5.1}(2) we have 	

$H^0(v',\alpha_{n-1})= \bigoplus\limits_{i=1}^{n-2} 
(\mathbb{C}_{-(\beta_{i} + 2\alpha_{n})} \oplus \mathbb{C}_{-(\beta_{i} + 2\alpha_{n}+ \beta_{n-1})} \oplus \cdots \oplus \mathbb{C}_{-(\beta_{i} + 2\alpha_{n} + \beta_{i+1})}) \oplus \mathbb{C}_{-(\beta_{n-1} + 2\alpha_{n})}
.$

Since $\{-(\beta_{i} + 2\alpha_{n}+\beta_{n-1}), \ldots, -(\beta_{i} + 2\alpha_{n} + \beta_{i+1}): 1\le i \le n-2\}$ are orthogonal to $\alpha_{n}$, by Lemma \ref {lemma1.3}(2) we have 

$H^0(\tilde{L}_{\alpha_{n}}/\tilde{B}_{\alpha_{n}}, \mathbb{C}_{-(\beta_{i} + 2\alpha_{n} + \beta_{l})})=\mathbb{C}_{-(\beta_{i} +
	2\alpha_{n} + \beta_{l})}$ for all  $ i+1 \le l \le n-1,$ where $1\le i \le n-2.$

Since 
$\langle -(\beta_{i}+ 2\alpha_{n}),\alpha_{n}\rangle=-2$ for all $1\le i\le n-1,$ by Lemma \ref {lemma1.3}(3) we have 

$H^0(\tilde{L}_{\alpha_{n}}/\tilde{B}_{\alpha_{n}}, \mathbb{C}_{-(\beta_{n-1} + 2\alpha_{n})})=0$ for all $1\le i \le n-1.$

From the above discussion, we have

$H^0(s_{n}v', \alpha_{n-1})=
\bigoplus\limits_{i=1}^{n-2} 
(\mathbb{C}_{-(\beta_{i} + 2\alpha_{n} + \beta_{n-1})} \oplus \cdots \oplus \mathbb{C}_{-(\beta_{i} + 2\alpha_{n} + \beta_{i+1})}).$

Since $\langle -(\beta_{i} + 2\alpha_{n} +\beta_{n-1}) , \alpha_{n-1} \rangle= -1$  for each $1\le i \le n-2,$ by Lemma \ref{lemma1.3}(4) we have  $H^0(\tilde{L}_{\alpha_{n-1}}/\tilde{B}_{\alpha_{n-1}}, \mathbb{C}_{-(\beta_{i} + 2\alpha_{n} + \beta_{n-1})})=0$  for each $1\le i\le n-2.$

Moreover, $\{-(\beta_{i} + 2\alpha_{n}+ \beta_{n-2}), \dots, -(\beta_{i} + 2\alpha_{n}+ \beta_{i+1}): 1\le i\le n-3\}$, and $\beta_{n-1}+2\alpha_{n}$  are orthogonal to $\alpha_{n-1}.$

Therefore  we have 

$H^0(s_{n-1}s_{n}v', \alpha_{n-1})= \bigoplus\limits_{i=1}^{n-3} 
(\mathbb{C}_{-(\beta_{i} + 2\alpha_{n} + \beta_{n-2})} \oplus \cdots \oplus \mathbb{C}_{-(\beta_{i} + 2\alpha_{n} + \beta_{i+1})})
.$

Proceeding recursively we have 

$H^0(s_{a_{k-1}}s_{a_{k-1}+1}\cdots s_{n}v', \alpha_{n-1})=
\bigoplus\limits_{i=1}^{a_{k-1}-2} 
(\mathbb{C}_{-(\beta_{i} + 2\alpha_{n} + \beta_{a_{k-1}-1})}\oplus \cdots \oplus \mathbb{C}_{-(\beta_{i} + 2\alpha_{n} + \beta_{i+1})}).$

 Since $\{-(\beta_{i} +2\alpha_{n} + \beta_{a_{k-2}-1}), \ldots , -(\beta_{i} + 2\alpha_{n} + \beta_{i+1}): 1\le i\le a_{k-1}-2\}$ are orthogonal to 
$\alpha_{j}$ for all $a_{k-2}\le j\le n$, we have

$H^0(s_{a_{k-2}}\cdots s_{n}s_{a_{k-1}}\cdots s_{n}v', \alpha_{n-1})=H^0(s_{a_{k-1}}\cdots s_{n}v', \alpha_{n-1}).$

Proceeding recursively we have 

$H^0(u_{1}, \alpha_{n-1})=
\bigoplus\limits_{i=1}^{a_{k-1}-2} 
(\mathbb{C}_{-(\beta_{i} + 2\alpha_{n} + \beta_{a_{k-1}-1})}\oplus \cdots \oplus \mathbb{C}_{-(\beta_{i} + 2\alpha_{n} + \beta_{i+1})}).$
\end{proof}
\begin{lem}\label{lemma 5.2.1} Let $3\le r \le k.$ Then $H^0(w_{r-2}s_{n}s_{a_{r-1}}s_{a_{r-1} + 1}\cdots s_{n + 2 - r}, \alpha_{n + 2 -r })_{\mu}\neq0$ if $\mu$ is of the form $\mu=-(\beta_{j} + \alpha_{n} )$ for some $a_{r-1}\le j\le a_{r-2} -1 .$  
	
\end{lem}
\begin{proof}
 By applying SES repeatedly, it is easy to see that

  $H^0(s_{n}s_{a_{r-1}}s_{a_{r-1}+1}\cdots s_{n+2-r}, \alpha_{n+2-r})=\mathbb{C}h({\alpha_{n+2-r}}) \oplus (\bigoplus\limits_{j=a_{r-1}}^{n+2-r}\mathbb{C}_{-\gamma_{j,n+2-r}}),$ where $\gamma_{j,j'}=(\alpha_{j}+\dots+\alpha_{j'})$ for $j'\ge j$. 

 Let $V_{1}=H^0(s_{n}s_{a_{r-1}}s_{a_{r-1}+1}\cdots s_{n+2-r}, \alpha_{n+2-r}).$ We next calculate $H^0(s_{a_{r-2}}\cdots s_{n-1}, V_{1}).$
 Since $n\ge a_{1}>a_{2}>\dots>a_{k}=1,$ we have $a_{i}\le n+1-i$ for all $1\le i\le k.$ Assume $l\ge n+4-r,$ then $\langle \gamma_{j, n+2-r} , \alpha_{l} \rangle=0$ for all $a_{r-1}\le j\le n+2-r.$ By Lemma \ref{lemma1.3}(2) we have $H^0(\tilde{L}_{\alpha_{l}}/\tilde{B}_{\alpha_{l}}, V_{1})=V_{1}$
 for all $l\ge n+4-r.$
 
 Therefore we have $H^0(s_{a_{r-2}}\cdots s_{n-1}, V_{1})=H^0(s_{a_{r-2}}\cdots s_{n+2-r}s_{n+3-r}, V_{1}).$
 
 Note that, since $\langle -\gamma_{j,n+2-r}, \alpha_{n+3-r}\rangle=1$ for all $a_{r-1}\le j\le n+2-r,$ by Lemma \ref{lemma1.3}(2) we have  $H^0(s_{n+3-r}, V_{1})=\mathbb{C}h({\alpha_{n+2-r}}) \oplus (\bigoplus\limits_{j=a_{r-1}}^{n+2-r}(\mathbb{C}_{-\gamma_{j, n+2-r}} \oplus \mathbb{C}_{-\gamma_{j, n+3-r}} )).$

Since $\mathbb{C}h({\alpha_{n+2-r}}) \oplus \mathbb{C}_{-\gamma_{n+2-r, n+2-r}}$ is  an indecomposable two dimensional $B_{\alpha_{n+2-r}}$-module, by Lemma \ref{lemma 1.4},
$\mathbb{C}h({\alpha_{n+2-r}}) \oplus \mathbb{C}_{-\gamma_{n+2-r, n+2-r}}=V\otimes \mathbb{C}_{-\omega_{n+2-r}},$ where $V$ is the irreducible two dimensional $\tilde{L}_{\alpha_{n+2-r}}$-module.

Since $\langle \gamma_{j, n+2-r}, \alpha_{n+2-r}\rangle=-1,$ for all $a_{r-1}\le j\le n+1-r,$ by Lemma \ref{lemma1.3}(4) we have 
 
$H^0(\tilde{L}_{\alpha_{n+2-r}}/\tilde{B}_{\alpha_{n+2-r}}, \mathbb{C}_{\gamma_{j, n+2-r}})=0$  for all $a_{r-1}\le j\le n+1-r.$

From the above discussion, we have 

$H^0(s_{n+2-r}, H^0(s_{n+3-r}, V_{1}))=\bigoplus\limits_{j=a_{r-1}}^{n+1-r}\mathbb{C}_{-\gamma_{j, n+3-r}}.$ 
 
Since $\langle \gamma_{n+1-r, n+3-r}, \alpha_{n+1-r}\rangle=-1,$ by Lemma \ref {lemma1.3}(4), we have 

$H^0(\tilde{L}_{\alpha_{n+1-r}}/\tilde{B}_{\alpha_{n+1-r}}, \mathbb{C}_{\gamma_{n+1-r, n+3-r}})=0.$

Moreover, $\langle \gamma_{j, n+3-r}, \alpha_{n+1-r}\rangle=0$ for all $a_{r-1}\le j\le n-r.$ 
Therefore we have

$H^0(s_{n+1-r}, H^0(s_{n+2-r}, H^0(s_{n+3-r}, V_{1})))=\bigoplus\limits_{j=a_{r-1}}^{n-r}\mathbb{C}_{-\gamma_{j, n+3-r}}.$

Proceeding recursively we have 

$H^0(s_{a_{r-2}}\cdots s_{n-1}, V_{1})=\bigoplus\limits_{j=a_{r-1}}^{a_{r-2}-1}\mathbb{C}_{-\gamma_{j, n+3-r}}$ and $H^0(s_{n}s_{a_{r-2}}\cdots s_{n-1}, V_{1})=\bigoplus\limits_{j=a_{r-1}}^{a_{r-2}-1}\mathbb{C}_{-\gamma_{j, n+3-r}}.$

Let $V_{2}= H^0(s_{n}s_{a_{r-2}}\cdots s_{n-1}, V_{1}).$ Similarly, we have $H^0(s_{a_{r-3}}\cdots s_{n-1}, V_{2})=\bigoplus\limits_{j=a_{r-1}}^{a_{r-2}-1}\mathbb{C}_{-\gamma_{j, n+4-r}}$ and $H^0(s_{n}s_{a_{r-3}}\cdots s_{n-1}, V_{2})=\bigoplus\limits_{j=a_{r-1}}^{a_{r-2}-1}\mathbb{C}_{-\gamma_{j, n+4-r}}.$ 

Proceeding recursively  we have 

$V_{r-2}=H^0(s_{a_{2}}\cdots s_{n-1}, V_{r-3})=\bigoplus\limits_{j=a_{r-1}}^{a_{r-2}-1}\mathbb{C}_{-\gamma_{j, n-1}}$, 

where $V_{r-3}=H^0((s_{n}s_{3}\cdots s_{n})\cdots(s_{a_{r-2}}\cdots s_{n+2-r}), \alpha_{n+2-r}).$

Note that $\gamma_{j, n-1}=\beta_{j}$. 
Since $\langle -\alpha_{n-1}, \alpha_{n} \rangle=2,$ by Lemma \ref{lemma1.3}(2) and Lemma \ref{lemma 1.4}, we have 

$$H^0(s_{n}, V_{r-2})=\bigoplus\limits_{j=a_{r-1}}^{a_{r-2}-1} (\mathbb{C}_{-\beta_{j}}\oplus \mathbb{C}_{-(\beta_{j}+\alpha_{n})}\oplus\mathbb{C}_{-(\beta_{j} +2\alpha_{n})}).$$

 Moreover, $\langle -\beta_{j}, \alpha_{n-1} \rangle=-1$, $\langle -(\beta_{j}+\alpha_{n}), \alpha_{n-1}\rangle=0$ and $\langle -(\beta_{j}+ 2\alpha_{n}), \alpha_{n-1}\rangle=1$ for all $a_{r-1}\le j\le a_{r-2}-1$. 
 
Therefore by Lemma \ref{lemma1.3}(2), Lemma \ref{lemma1.3}(4) and Lemma \ref{lemma 1.4} we have

$$H^0(s_{n-1}, H^0(s_{n}, V_{r-2}))=\bigoplus\limits_{j=a_{r-1}}^{a_{r-2}-1}(\mathbb{C}_{-(\beta_{j}+\alpha_{n})}\oplus \mathbb{C}_{-(\beta_{j}+2\alpha_{n})}\oplus \mathbb{C}_{-(\beta_{j}+2\alpha_{n}+\beta_{n-1})}).$$

Proceeding recursively we have 

$H^0(s_{a_{1}}, H^0(s_{a_{1}+1}\cdots s_{n}, V_{r-2}))=\bigoplus\limits_{j=a_{r-1}}^{a_{r-2}-1}(\mathbb{C}_{-(\beta_{j}+\alpha_{n})} \oplus \mathbb{C}_{-(\beta_{j}+2\alpha_{n})} \oplus \mathbb{C}_{-(\beta_{j} + 2\alpha_{n} +\beta_{n-1})} \oplus \dots \oplus \mathbb{C}_{-(\beta_{j}+2\alpha_{n}+\beta_{a_{1}})}).$

Hence the proof of the lemma follows.

\end{proof}

\begin{lem}\label{lemma 5.3}
	
Let $3\le r \le k.$ Then $H^0(w_{r-1}s_{n}, \alpha_{n})_{\mu}\neq 0$ if and only if $\mu$ is of the form $\mu=-(\beta_{j} + \alpha_{n}),$ for some $a_{r-1}\le j \le a_{r-2} - 1.$ 

\end{lem}
\begin{proof}
 By applying SES repeatedly, it is easy to see that

 $H^0(s_{a_{r-1}}\cdots s_{n-1}s_{n}, \alpha_{n})=\bigoplus\limits_{j=a_{r-1}}^{n-1}\mathbb{C}_{-(\beta_{j} + \alpha_{n})}.$

Let $V_{1}=H^0(s_{a_{r-1}}\cdots s_{n-1}s_{n}, \alpha_{n}).$
 Since $\langle -(\beta_{j}+\alpha_{n}) , \alpha_{n} \rangle=0,$ we have $H^0(s_{n}, V_{1})=V_{1}.$

Since $\langle -(\beta_{n-1}+\alpha_{n}), \alpha_{n-1} \rangle=-1$
and  $\langle -(\beta_{j}+\alpha_{n}, \alpha_{n-1})\rangle=0,$ for all $a_{r-1}\le j\le n-2,$ by Lemma \ref{lemma1.3}(2), Lemma \ref{lemma1.3}(4) and Lemma \ref{lemma 1.4} we have 

$$H^0(s_{n-1}, V_{1})=\bigoplus\limits_{j=a_{r-1}}^{n-2}\mathbb{C}_{-(\beta_{j} +\alpha_{n})}.$$

Proceeding recursively we have 

$$H^0(s_{a_{r-2}}\cdots s_{n-1}s_{n}, V_{1})=\bigoplus\limits_{j=a_{r-1}}^{a_{r-2}-1}\mathbb{C}_{-(\beta_{j} +\alpha_{n})}.$$
 
Since $n\ge a_{1}>a_{2}>\dots>a_{k}=1$, we see that $\langle -(\beta_{j}+ \alpha_{n}), \alpha_{t}\rangle=0$,  for all $a_{r-1}\le j\le a_{r-2}-1$ and for all $a_{r-3}\le t\le n,$ therefore by Lemma \ref{lemma1.3}(2) and Lemma \ref{lemma 1.4}, we have

$$H^0(s_{a_{r-3}}\cdots s_{n}s_{a_{r-2}}\cdots s_{n}s_{a_{r-1}}\cdots s_{n-1}s_{n}, \alpha_{n})=\bigoplus\limits_{j=a_{r-1}}^{a_{r-2}-1}\mathbb{C}_{-(\beta_{j} +\alpha_{n})}.$$ 

 Since $n\ge a_{1}>a_{2}>\dots>a_{k}=1$, we see that  $\langle -(\beta_{j}+ \alpha_{n}), \alpha_{t}\rangle=0$ for all $a_{r-1}\le j\le a_{r-2}-1$ and for all $a_{r-4}\le t\le n.$ By Lemma \ref{lemma1.3}(2) and Lemma \ref{lemma 1.4}, we have

$$H^0(s_{a_{r-4}}\cdots s_{n}s_{a_{r-3}}\cdots s_{n}s_{a_{r-2}}\cdots s_{n}s_{a_{r-1}}\cdots s_{n-1}s_{n}, \alpha_{n})=\bigoplus\limits_{j=a_{r-1}}^{a_{r-2}-1}\mathbb{C}_{-(\beta_{j} +\alpha_{n})}.$$

Proceeding recursively we have 

$$H^0(w_{r-1}s_{n}, \alpha_{n})=\bigoplus\limits_{j=a_{r-1}}^{a_{r-2}-1}\mathbb{C}_{-(\beta_{j} +\alpha_{n})}.$$

Hence the lemma follows. 
\end{proof}

\begin{lem}\label{lemma 5.4} If $\mu$ is of the form $\mu=-(\beta_{j} + \alpha_{n} )$ for some $1 \le j\le a_{k-1} -1 ,$ then we have $H^0(w_{k-1}s_{n}s_{1}s_{2}\cdots s_{n + 1 - k}, \alpha_{n + 1-k })_{\mu}\neq0.$
	
\end{lem}
\begin{proof}
By applying SES repeatedly, it is easy to see that 
	
	$H^0(s_{n}s_{a_{1}}s_{2}\cdots s_{n+1-k}, \alpha_{n+1-k})=\mathbb{C}h({\alpha_{n+1-k}}) \oplus (\bigoplus\limits_{j=1}^{n+1-k}\mathbb{C}_{-\gamma_{j,n+1-k}}),$
	
where $\gamma_{j,j'}=(\alpha_{j}+\dots+\alpha_{j'})$ for $j'\ge j$.

Let $V_{1}=H^0(s_{n}s_{1}s_{2}\cdots s_{n+1-k}, \alpha_{n+1-k}).$ We next calculate $H^0(s_{a_{k-1}}\cdots s_{n-1}, V_{1}).$

Since $n\ge a_{1}>a_{2}>\dots>a_{k}=1,$ we have $a_{i}\le n+1-i$ for all $1\le i\le k.$
Moreover, $\langle \gamma_{j, n+1-k} , \alpha_{l} \rangle=0$ for all $1\le j\le n+1-k$  and for all $l\ge n+3-k.$

Therefore by using Lemma \ref {lemma1.3}(2) and Lemma \ref{lemma 1.4}, we have 

 $$H^0(s_{a_{k-1}}\cdots s_{n-1}, V_{1})=H^0(s_{a_{k-1}}\cdots s_{n+1-k}s_{n+2-k}, V_{1}).$$

Since $\langle  -\gamma_{j, n+1-k},\alpha_{n+2-k}\rangle=1$ for all $1\le j\le n+1-k,$ by  using Lemma \ref {lemma1.3}(2) and Lemma \ref{lemma 1.4}, we have  

$H^0(s_{n+2-k}, V_{1})=\mathbb{C}h({\alpha_{n+1-k}}) \oplus (\bigoplus\limits_{j=1}^{n+1-k}(\mathbb{C}_{-\gamma_{j,n+1-k}} \oplus \mathbb{C}_{-\gamma_{j, n+2-k}} )).$
	
Since $\mathbb{C}h({\alpha_{n+1-k}}) \oplus \mathbb{C}_{-\gamma_{n+1-k, n+1-k}}$ is an indecomposable two dimensional $B_{\alpha_{n+1-k}}$-module, by Lemma \ref{lemma 1.4} we have $\mathbb{C}h({\alpha_{n+1-k}}) \oplus \mathbb{C}_{-\gamma_{n+1-k, n+1-k}}=V\otimes \mathbb{C}_{-\omega_{n+1-k}}$, where $V$ is the irreducible two dimensional $\tilde{L}_{\alpha_{n+1-k}}$-module.

 By Lemma \ref{lemma1.3}(4) we have 

$H^0(\tilde{L}_{\alpha_{n+1-k}}/\tilde{B}_{\alpha_{n+1-k}}, \mathbb{C}h({\alpha_{n+1-k}})\oplus\mathbb{C}_{-\gamma_{n+1-k, n+1-k}})=0.$

Since $\langle \gamma_{j, n+1-k}, \alpha_{n+1-k}\rangle=-1$ for all $1\le j\le n-k,$ by Lemma \ref{lemma1.3}(4) we have

$H^0(\tilde{L}_{\alpha_{n+1-k}}/\tilde{B}_{\alpha_{n+1-k}}, \mathbb{C}_{\gamma_{j, n+1-k}})=0$ for all $1\le j\le n-k.$ 

From the above discussion, we have
	
	$H^0(s_{n+1-k}, H^0(s_{n+2-k}, V_{1}))=\bigoplus\limits_{j=1}^{n-k}\mathbb{C}_{-\gamma_{j,n+2-k}}.$ 
	
Since $\langle \gamma_{j, n+2-k}, \alpha_{n-k}\rangle=0$ for all $1\le j\le n-k-1,$ and  $\langle \gamma_{n-k, n+2-k}, \alpha_{n-k}\rangle=-1,$ by using  Lemma \ref {lemma1.3}(2), Lemma \ref {lemma1.3}(2)(4) and Lemma \ref{lemma 1.4}, we have

	$$H^0(s_{n-k}, H^0(s_{n+1-k}, H^0(s_{n+2-k}, V_{1})))=\bigoplus\limits_{j=1}^{n-k-1}\mathbb{C}_{-\gamma_{j,n+2-k}}.$$
	
Proceeding recursively we have

 $$H^0(s_{n}s_{a_{k-1}}\cdots s_{n-1}, V_{1})=\bigoplus\limits_{j=1}^{a_{k-1}-1}\mathbb{C}_{-\gamma_{j,n+2-k}}.$$
	
Let $V_{2}= H^0(s_{n}s_{a_{k-1}}\cdots s_{n-1}, V_{1}).$ Then similarly, we have

	$$H^0(s_{a_{k-2}}\cdots s_{n-1}, V_{2})=\bigoplus\limits_{j=1}^{a_{k-1}-1}\mathbb{C}_{-\gamma_{j, n+3-k}}.$$
	
Proceeding recursively we have
	
	$$V_{k-1}=H^0(s_{a_{2}}\cdots s_{n-1}, V_{k-2})=\bigoplus\limits_{j=1}^{a_{k-1}-1}\mathbb{C}_{-\gamma_{j, n-1}}$$
	
	where $V_{k-2}=H^0((s_{n}s_{3}\cdots s_{n})\cdots(s_{a_{k-1}}\cdots s_{n+1-k}), \alpha_{n+1-k}).$
	
	Note that $\gamma_{j, n-1}=\beta_{j}$. Since $\langle -\alpha_{n-1}, \alpha_{n} \rangle= 2,$ by using Lemma \ref {lemma1.3}(2) and Lemma \ref {lemma 1.4} we have 
	
	$H^0(s_{n}, V_{k-1})=\bigoplus\limits_{j=1}^{a_{k-1}-1} (\mathbb{C}_{-\beta_{j}}\oplus \mathbb{C}_{-(\beta_{j}+\alpha_{n})}\oplus\mathbb{C}_{-(\beta_{j} +2\alpha_{n})}).$
	
 Moreover, $\langle -\beta_{j}, \alpha_{n-1} \rangle=-1,$  $\langle -(\beta_{j}+\alpha_{n}), \alpha_{n-1}\rangle=0$ and $\langle -(\beta_{j}+ 2\alpha_{n}), \alpha_{n-1}\rangle=1$, for all $1\le j\le a_{k-1}-1.$ 
 
 Therefore  by using  Lemma \ref {lemma1.3}(2), Lemma \ref {lemma1.3}(4) and Lemma \ref {lemma 1.4} we have 
	
	$$H^0(s_{n-1}, H^0(s_{n}, V_{k-1}))=\bigoplus\limits_{j=1}^{a_{k-1}-1}(\mathbb{C}_{-(\beta_{j}+\alpha_{n})}\oplus \mathbb{C}_{-(\beta_{j}+2\alpha_{n})}\oplus \mathbb{C}_{-(\beta_{j}+2\alpha_{n}+\beta_{n-1})}).$$
	
Proceeding recursively we have 
	
	$H^0(s_{a_{1}}, H^0(s_{a_{1}+1}\cdots s_{n}, V_{k-1}))=\bigoplus\limits_{j=1}^{a_{k-1}-1}(\mathbb{C}_{-(\beta_{j}+\alpha_{n})}\oplus\mathbb{C}_{-(\beta_{j}+2\alpha_{n})}\oplus\mathbb{C}_{-(\beta_{j} + 2\alpha_{n} +\beta_{n-1})} \oplus \dots \oplus \mathbb{C}_{-(\beta_{j}+2\alpha_{n}+\beta_{a_{1}})}).$
	
	Hence the proof of the lemma follows. 
	
\end{proof}

\begin{lem}\label{lemma 5.5}
	
 $H^0(w_{k}s_{n}, \alpha_{n})_{\mu}\neq 0$ if and only if $\mu$ is of the form $\mu=-(\beta_{j} + \alpha_{n})$ for some $1\le j \le a_{k-1} - 1.$ 
	
\end{lem}
\begin{proof}
By applying SES repeatedly it is easy to see that 
	
	$$H^0(s_{1}\cdots s_{n-1}s_{n}, \alpha_{n})=\bigoplus\limits_{j=1}^{n-1}\mathbb{C}_{-(\beta_{j} + \alpha_{n})}.$$
	
Let $V_{1}=H^0(s_{1}\cdots s_{n-1}s_{n}, \alpha_{n}).$
Since  $\langle -(\beta_{j}+\alpha_{n}) , \alpha_{n} \rangle=0,$ for all $1\le j\le n,$ by Lemma \ref {lemma1.3}(2) and Lemma \ref {lemma 1.4} we have $H^0(s_{n}, V_{1})=V_{1}.$
	
Moreover, $\langle -(\beta_{n-1}+\alpha_{n}), \alpha_{n-1} \rangle=-1$ and $\langle -(\beta_{j}+\alpha_{n}, \alpha_{n-1})\rangle=0$ for all $1\le j\le n-2.$ 
Therefore by using Lemma \ref {lemma1.3}(4), Lemma \ref {lemma1.3}(2) and Lemma \ref {lemma 1.4}, we have
	
	$$H^0(s_{n-1}, V_{1})=\bigoplus\limits_{j=1}^{n-2}\mathbb{C}_{-(\beta_{j} +\alpha_{n})}.$$
	
	Proceeding recursively we have

	$$H^0(s_{a_{k-1}}\cdots s_{n-1}s_{n}, V_{1})=\bigoplus\limits_{j=1}^{a_{k-1}-1}\mathbb{C}_{-(\beta_{j} +\alpha_{n})}.$$
	
Since $n\ge a_{1}>a_{2}>\dots>a_{k}=1$, we have $\langle -(\beta_{j}+ \alpha_{n}), \alpha_{t}\rangle=0$ for all $1\le j\le a_{k-1}-1$ and for all $a_{k-2}\le t\le n.$ 
	
By using Lemma \ref {lemma1.3}(2), Lemma \ref {lemma 1.4} we have
	
	$$H^0(s_{a_{k-2}}\cdots s_{n}s_{a_{k-1}}\cdots s_{n}s_{1}\cdots s_{n-1}s_{n}, \alpha_{n})=\bigoplus\limits_{j=1}^{a_{k-1}-1}\mathbb{C}_{-(\beta_{j} +\alpha_{n})}.$$ 
	
Similarly, since $n\ge a_{1}>a_{2}>\dots>a_{k}=1$, we have $\langle -(\beta_{j}+ \alpha_{n}), \alpha_{t}\rangle=0$ for all $1\le j\le a_{k-1}-1$ and for all $a_{k-3}\le t\le n,$ therefore 
by using Lemma \ref {lemma1.3}(2) and Lemma \ref {lemma 1.4} we have
	
	$$H^0(s_{a_{k-3}}\cdots s_{n}s_{a_{k-2}}\cdots s_{n}s_{a_{k-1}}\cdots s_{n}s_{1}\cdots s_{n-1}s_{n}, \alpha_{n})=\bigoplus\limits_{j=1}^{a_{k-1}-1}\mathbb{C}_{-(\beta_{j} +\alpha_{n})}.$$
	
Proceeding recursively we have 
	
	$$H^0(w_{k}s_{n},  \alpha_{n})=\bigoplus\limits_{j=1}^{a_{k-1}-1}\mathbb{C}_{-(\beta_{j} +\alpha_{n})}.$$
	
Hence the proof of the lemma follows. 
\end{proof}

\section{cohomology module $H^1$ of the relative tangent bundle}
In this section, we describe the weights of $H^1$ of the relative tangent bundle. Let $n\ge a_{1}>a_{2}>\ldots >a_{k-1}> a_{k}=1$ be a decreasing sequence of integers such that $k\ge 3.$ Fix $3\le r \le k.$
\begin{lem}\label{lemma 6.1} Let $v_{r}=s_{n}s_{a_{r}}\cdots s_{n-2},$ $v_{r-1}=s_{a_{r-1}}\cdots s_{n-1}$ and $v_{r-2}=s_{a_{r-2}}\cdots s_{n-1}s_{n}.$ Then we have
	
\begin{enumerate}

\item [(1)] $H^1(v_{r-1}v_{r}s_{n-1}, \alpha_{n-1}) =0.$ 

\item [(2)] Let $w=v_{r-2}v_{r-1}v_{r}s_{n-1}.$  $H^1(w, \alpha_{n-1})_{\mu}\neq 0$ if and only if $\mu$ is of the form

 $\mu=-(\beta_{t} + \alpha_{n})$ for some $a_{r-1}\le t \le a_{r-2}-1.$ In such a case,\\  dim$(H^1(w, \alpha_{n-1})_{\mu})=1.$ 

\end{enumerate}
\end{lem}
\begin{proof}
Proof of $(1)$: Note that $H^0(s_{n-1}, \alpha_{n-1})=\mathbb{C}_{-\alpha_{n-1}} \oplus \mathbb{C}h({\alpha_{n-1}}) \oplus \mathbb{C}_{\alpha_{n-1}}$  (see \cite[Corollary 2.5]{CKP}). \hspace{13cm} (5.1.1)

 Since $\langle \alpha_{n-1}, \alpha_{n-1} \rangle = 2,$ we have $H^1(s_{n-1}, \alpha_{n-1})=0.$ Since $\langle \alpha_{n-1} , \alpha_{n-2} \rangle=-1$, we have
 $H^1(s_{n-2},H^0(s_{n-1}, \alpha_{n-1}))=0.$

Therefore by using SES we have  $H^1(s_{n-2}s_{n-1}, \alpha_{n-1})=0.$    \hspace{4.5cm} (5.1.2)

Since $\langle -\alpha_{n-1} ,\alpha_{n-2} \rangle=1$, by using (5.1.1) we have 

$H^0(s_{n-2}s_{n-1}, \alpha_{n-1})=\mathbb{C}h({\alpha_{n-1}}) \oplus \mathbb{C}_{-\alpha_{n-1}} \oplus \mathbb{C}_{-\beta_{n-2}}.$

Since $\langle -\alpha_{n-1} , \alpha_{n-3} \rangle=0$ and  $\langle -\beta_{n-2} , \alpha_{n-3} \rangle=1$ we have

 $H^1(s_{n-3}, H^0(s_{n-2}s_{n-1}, \alpha_{n-1}))=0.$  \hspace{8.4cm}(5.1.3)

Therefore by using SES together with  (5.1.2), (5.1.3) we have 

$H^1(s_{n-3}s_{n-2}s_{n-1}, \alpha_{n-1})=0.$  \hspace{9.4cm}(5.1.4)

Proceeding in this way we have $H^1(s_{a_{r}}\cdots s_{n-2}s_{n-1}, \alpha_{n-1})=0$       \hspace{3cm}    (5.1.5)
and $H^0(s_{a_{r}}\cdots s_{n-2}s_{n-1}, \alpha_{n-1})=\mathbb{C}h(\alpha_{n-1}) \oplus \bigoplus\limits_{j=a_{r}}^{n-1}\mathbb{C}_{-\beta_{j}}.$ \hspace{5cm}    (5.1.6)

Since $\langle - \beta_{j}, \alpha_{n} \rangle>0$ for all $a_{r}\le j\le n-1,$ by using (5.1.6) we have 

$H^1(s_{n}, H^0(s_{a_{r}}\cdots s_{n-2}s_{n-1}, \alpha_{n-1}))=0.$  \hspace{7.6cm} (5.1.7)

 Therefore by using SES, (5.1.5) and (5.1.7) together  we have

  $H^1(v_{r}s_{n-1}, \alpha_{n-1})=0.$ \hspace{10.7cm}(5.1.8)

 In the proof of Lemma \ref {lemma 5.1}(1) we notice that
 
  $H^0(v_{r}s_{n-1}, \alpha_{n-1})= \mathbb{C}h({\alpha_{n-1}})\oplus (\bigoplus\limits_{j=a_{r}}^{n-1}\mathbb{C}_{\beta_{j}}\oplus \mathbb{C}_{-(\beta_{j} + \alpha_{n})}\oplus \mathbb{C}_{-(\beta_{j} + 2\alpha_{n})}).$      
  
  Thus we have 
  
  $H^1(s_{n-1}, H^0(v_{r}s_{n-1}, \alpha_{n-1}))=0$ (see lines from (4.1.1) to (4.1.2)).\hspace{3cm} (5.1.9)

Proceeding by similar arguments and using (5.1.8), (5.1.9) we have

 $H^1(v_{r-1}v_{r}s_{n-1}, \alpha_{n-1})=0.$
 
Proof of (2) 

From the Lemma \ref {lemma 5.1}(1) we have 

$H^0(v_{r-1}v_{r}s_{n-1}, \alpha_{n-1})=\bigoplus\limits_{i=a_{r}}^{a_{r-1}-1}(\mathbb{C}_{-(\beta_{i} + \alpha_{n})} \oplus \mathbb{C}_{-(\beta_{i} + 2\alpha_{n})}\oplus\mathbb{C}_{-(\beta_{i} + 2\alpha_{n} + \beta_{n-1})} \oplus \cdots \oplus\mathbb{C}_{-(\beta_{i} + 2\alpha_{n} +
	\beta_{a_{r-1}})})$

$\bigoplus\limits_{i=a_{r-1}}^{n-2} 
(\mathbb{C}_{-(\beta_{i} + 2\alpha_{n})}\oplus \mathbb{C}_{-(\beta_{i} + 2\alpha_{n}+\beta_{n-1})}\oplus \cdots \oplus \mathbb{C}_{-(\beta_{i} + 2\alpha_{n} +
\beta_{i+1})})\oplus \mathbb{C}_{-(\alpha_{n-1} + 2\alpha_{n})}
.$

Notice that for each $a_{r}\le i\le a_{r-1}-1$, $\mathbb{C}_{-(\beta_{i} + \alpha_{n})} \oplus \mathbb{C}_{-(\beta_{i} + 2\alpha_{n})}$ forms an indecomposable two dimensional $B_{\alpha_{n}}$-module.

Since $\langle -(\beta_{i} + 2\alpha_{n}), \alpha_{n} \rangle=-2,$
by Lemma \ref{lemma 1.4} we have 

$\mathbb{C}_{-(\beta_{i} + \alpha_{n})} \oplus \mathbb{C}_{-(\beta_{i} + 2\alpha_{n})}=V_{i}\otimes \mathbb{C}_{-\omega_{n}},$ where $V_{i}$ is the irreducible two dimensional $\tilde{L}_{\alpha_{n}}$-module.

By Lemma \ref{lemma1.3}(4) we have 

$H^j(\tilde{L}_{\alpha_{n}}/\tilde{B}_{\alpha_{n}}, \mathbb{C}_{-(\beta_{i} + \alpha_{n})} \oplus \mathbb{C}_{-(\beta_{i} + 2\alpha_{n})})=0$ for all $j\ge 0$ and for all $a_{r}\le i\le a_{r-1}-1.$

Since $\langle -(\beta_{i} + 2\alpha_{n} + \beta_{t}), \alpha_{n}\rangle=0$ for each $a_{r}\le i\le a_{r-1}-1$, and $a_{r-1}\le t\le n-1,$ we have 

$H^1(\tilde{L}_{\alpha_{n}}/\tilde{B}_{\alpha_{n}}, \mathbb{C}_{-(\beta_{i} + 2\alpha_{n} + \beta_{t})})=0$ for all $a_{r}\le i\le a_{r-1}-1$  and $a_{r-1}\le t\le n-1.$

Moreover, we have $\langle -(\beta_{i} + 2\alpha_{n} +
\beta_{i+1}), \alpha_{n}\rangle=0$ for each $a_{r-1}\le i\le n-2$ and $\langle -(\beta_{i} + 2\alpha_{n}), \alpha_{n}\rangle=-2$ for each $a_{r-1}\le i \le n-1$.

Therefore we have

 $H^1(\tilde{L}_{\alpha_{n}}/\tilde{B}_{\alpha_{n}}, \mathbb{C}_{-(\beta_{i} + 2\alpha_{n} + \beta_{i+1})})=0$ for all $a_{r-1}\le i\le n-2.$ 
 
By Lemma \ref {lemma1.3}(3) we have

 $H^1(\tilde{L}_{\alpha_{n}}/\tilde{B}_{\alpha_{n}}, \mathbb{C}_{-(\beta_{i} + 2\alpha_{n})})=H^0(\tilde{L}_{\alpha_{n}}/\tilde{B}_{\alpha_{n}}, s_{n-1}\cdot-(\beta_{i} + 2\alpha_{n}))=\mathbb{C}_{-(\beta_{i} + \alpha_{n})}$ for all $a_{r-1}\le i\le n-1.$

 From the above discussion, we have 

$H^1(s_{n}, H^0(v_{r-1}v_{r}s_{n-1}, \alpha_{n-1}))=\bigoplus\limits_{i=a_{r-1}}^{n-1}\mathbb{C}_{-(\beta_{i} + \alpha_{n})}.$  \hspace{7cm}(5.2.1)

By (1) and using SES we have 

$ H^1(s_{n}, H^0(v_{r-1}v_{r}s_{n-1},  \alpha_{n-1}))= H^1(s_{n}v_{r-1}v_{r}s_{n-1}, \alpha_{n-1}).$ \hspace{5.4cm}  (5.2.2)

From (5.2.1) and (5.2.2) we have 

$$H^1(s_{n}v_{r-1}v_{r}s_{n-1}, \alpha_{n-1})=\bigoplus\limits_{i=a_{r-1}}^{n-1}\mathbb{C}_{-(\beta_{i} + \alpha_{n})}. \hspace{8.5cm} (5.2.3)$$      

By Lemma \ref {lemma 2} we have

 $H^1(s_{n-1}, H^0(s_{n}v_{r-1}v_{r}s_{n-1}, \alpha_{n-1}))=0.$ \hspace{8.6cm}(5.2.4)

Let $v=v_{r-1}v_{r}s_{n-1}.$
Therefore by using SES and (5.2.4) we have 

$H^1(s_{n-1}s_{n}v, \alpha_{n-1})= H^0(s_{n-1}, H^1(s_{n}v, \alpha_{n-1})). $

Since $\langle -(\beta_{n-1} + \alpha_{n}), \alpha_{n-1} \rangle=-1 $ and $\langle -(\beta_{i} + \alpha_{n}), \alpha_{n-1} \rangle=0 $ for all $a_{r-1}\le i\le n-2,$ by using (5.2.3) we have

 $H^1(s_{n-1}s_{n}v, \alpha_{n-1})=\bigoplus\limits_{i=a_{r-1}}^{n-2}\mathbb{C}_{-(\beta_{i} + \alpha_{n})}.$

Proceeding recursively we have

$H^1(v_{r-2}v_{r-1}v_{r}s_{n-1} , \alpha_{n-1})=H^1(s_{a_{r-2}}\cdots s_{n-1}s_{n}v, \alpha_{n-1})=\bigoplus\limits_{i=a_{r-1}}^{a_{r-2}-1}\mathbb{C}_{-(\beta_{i} + \alpha_{n})}.$

\end{proof}

Recall that $w_{r}=[a_{1}, n][a_{2}, n]\cdots [a_{r-1},n][a_{r}, n-1],$ where $1 \le r \le k$ and $n\ge a_{1}>a_{2}>\ldots >a_{k-1}>a_{k}=1.$ 
\begin{lem}\label{lemma 6.2}
	
\begin{enumerate}
\item[(1)] $H^1(w_{1}, \alpha_{n-1})=0.$
\item[(2)] If $a_{2}\neq n-1,$ then $H^1(w_{2}, \alpha_{n-1})=0.$
\item[(3)] Let $3\le r \le k.$ Then, $H^1(w_{r}, \alpha_{n-1})_{\mu}\neq 0$ if and only if $\mu=$-$(\beta_{j} + \alpha_{n})$ for some $j$ such that $a_{r-1}\le j\le a_{r-2}-1.$ In such case dim$(H^1(w_{r}, \alpha_{n-1}))_{\mu}=1.$
\end{enumerate}
\end{lem}
\begin{proof}

Proof of (1):  Follows from proof of Lemma \ref {lemma 2} and using SES.
	
Proof of (2): By proof of (1), we have 

$H^1(s_{a_{2}}\cdots s_{n-1}, \alpha_{n-1})=0.$ \hspace{10.3cm} (5.2.1)
Since $a_{2}\neq n-1,$ $H^0(s_{a_{2}}\cdots s_{n-1}, \alpha_{n-1})=\mathbb{C}h(\alpha_{n-1})\oplus (\bigoplus\limits_{j=a_{2}}^{n-1}\mathbb{C}_{-(\beta_{j})}).$

Since $\langle -\beta_{j}, \alpha_{n}\rangle \ge 1$ for all $a_{2}\le j\le n-1,$ we have 

$H^1(s_{n}, H^0(s_{a_{2}}\cdots s_{n-1}, \alpha_{n-1}))=0.$ \hspace{9cm}(5.2.2)

By SES, (5.2.1), (5.2.2) we have $H^1(s_{n}s_{a_{2}}\cdots s_{n-1}, \alpha_{n-1})=0.$

By using Lemma \ref {lemma 2} and SES respeatedly, we see that 

$H^1(w_{2}, \alpha_{n-1})=H^1(s_{a_{1}}\cdots s_{n}s_{a_{2}}\cdots s_{n-1}, \alpha_{n-1})=0.$

Proof of (3): 

Let $v_{r}=s_{n}s_{a_{r}}\cdots s_{n-2},$ $v_{r-1}=s_{a_{r-1}}\cdots s_{n-1}$ and $v_{r-2}=s_{a_{r-2}}\cdots s_{n-1}s_{n}.$ Then by Lemma \ref {lemma 6.1}(2) we have

$$H^1(v_{r-2}v_{r-1}v_{r}s_{n-1}, \alpha_{n-1})=\bigoplus\limits_{j=a_{r-1}}^{a_{r-2}-1}\mathbb{C}_{-(\beta_{j} + \alpha_{n})}.$$  

 By Lemma \ref {lemma 5.2}(1) if $(H^0(v_{r-2}v_{r-1}v_{r}s_{n-1}, \alpha_{n-1}))_{\mu}\neq0$ then $\langle\mu, \alpha_{n} \rangle\ge0.$
 
 Therefore we have $H^1(s_{n}, H^0(v_{r-2}v_{r-1}v_{r}s_{n-1}, \alpha_{n-1}) )=0.$    \hspace{5cm}(5.3.1)

By using SES and (5.3.1) we have 

$H^1(s_{n}v_{r-2}v_{r-1}v_{r}s_{n-1} , \alpha_{n-1})= H^0(s_{n}, H^1(v_{r-2}v_{r-1}v_{r}, \alpha_{n-1})).$       \hspace{4.5cm}     (5.3.2)

Since $\langle -(\beta_{j} + \alpha_{n}), \alpha_{n}\rangle=0$ for all $a_{r-1}\le j\le a_{r-2}-1$, we have from (5.3.2)

$H^1(s_{n}v_{r-2}v_{r-1}v_{r}s_{n-1} , \alpha_{n-1})=\bigoplus\limits_{j=a_{r-1}}^{a_{r-2}-1}\mathbb{C}_{-(\beta_{j} + \alpha_{n})}.$

By Lemma \ref {lemma 2}, we have $H^1(s_{n-1}, H^0(s_{n}v_{r-2}v_{r-1}v_{r}s_{n-1}, \alpha_{n-1}) )=0.$\hspace{3.5cm}  (5.3.4)

Therefore by using SES together with (5.3.4) we have

$H^1(s_{n-1}s_{n}v_{r-2}v_{r-1}v_{r} , \alpha_{n-1})= H^0(s_{n-1}, H^1(s_{n}v_{r-2}v_{r-1}v_{r}s_{n-1}, \alpha_{n-1})).$ \hspace{3.0cm}(5.3.5)

Since $\langle -(\beta_{j} + \alpha_{n}), \alpha_{n-1}\rangle=0$ for all $a_{r-1}\le j\le a_{r-2}-1,$ using (5.3.5) we have 

$H^1(s_{n-1}s_{n}v_{r-2}v_{r-1}v_{r}s_{n-1}, \alpha_{n-1})=\bigoplus\limits_{j=a_{r-1}}^{a_{r-2}-1}\mathbb{C}_{-(\beta_{j} + \alpha_{n})}.$

Proceeding recursively  we have

$$H^1(s_{a_{r-3}}\cdots s_{n-1}s_{n}v_{r-2}v_{r-1}v_{r}s_{n-1} , \alpha_{n-1})=\bigoplus\limits_{j=a_{r-1}}^{a_{r-2}-1}\mathbb{C}_{-(\beta_{j} + \alpha_{n})}.$$

Since $\langle -(\beta_{j} + \alpha_{n}), \alpha_{t}\rangle$ for all $a_{r-1}\le j\le a_{r-2}-1$ and $a_{r-4}\le t\le n,$ using similar arguments as above we have  

$$H^1(w_{r}, \alpha_{n-1})=\bigoplus\limits_{j=a_{r-1}}^{a_{r-2}-1}\mathbb{C}_{-(\beta_{j} + \alpha_{n})}.$$

\end{proof}
\begin{cor}\label{cor 6.3}
Let $3\le r\le k.$ If  $H^1(w_{r}, \alpha_{n-1})_{\mu}\neq 0,$ then

 $H^0(w_{r-2}s_{n}s_{a_{r-1}}\cdots s_{n + 2-r}, \alpha_{n+2-r})_{\mu}\neq0.$
\end{cor}
\begin{proof}
Corollary follows from Lemma \ref {lemma 5.2.1} and Lemma \ref{lemma 6.2}(2). 
\end{proof}
\begin{cor}\label{cor 6.4}
Let $3\le r \le k.$ If $H^1(w_{r}, \alpha_{n-1})_{\mu}\neq 0,$ then $H^0(w_{r-1}s_{n}, \alpha_{n})_{\mu}\neq0.$
\end{cor}
\begin{proof}
Corollary follows from Lemma \ref{lemma 5.3} and Lemma
\ref{lemma 6.2}(2).	
\end{proof}

\begin{lem}\label{lemma 6.5} Let $u_{1}=w_{k}s_{n}[a_{k}, n-1].$ Then 
	$H^1(u_{1}, \alpha_{n-1})_{\mu}\neq 0$ if and only if $\mu$ is of the form, $\mu=-(\beta_{j} + \alpha_{n})$ for some $1\le j \le a_{k-1}-1.$	
\end{lem}
\begin{proof}

Let $v_{k+1}=s_{n}s_{1}\cdots s_{n-2},$ $v_{k}=s_{1}\cdots s_{n-1}$ and $v_{k-1}=s_{a_{k-1}}\cdots s_{n-1}s_{n}.$ 

{\bf Step 1:} 

$$H^1(v_{k-1}v_{k}v_{k+1}s_{n-1}, \alpha_{n-1})=\bigoplus\limits_{j=1}^{a_{k-1}-1}\mathbb{C}_{-(\beta_{j} + \alpha_{n})}.$$  
Proof of Step 1: By Lemma \ref {lemma 5.1}(2), we have

$H^0(s_{1}s_{2}\cdots s_{n}s_{1} \cdots s_{n-1}, \alpha_{n-1})=\bigoplus\limits_{i=1}^{n-2}(\mathbb{C}_{-(\beta_{i} + 2\alpha_{n})}\oplus \mathbb{C}_{-(\beta_{i} + 2\alpha_{n} + \beta_{n-1})}\oplus \dots \oplus \mathbb{C}_{-(\beta_{i} + 2\alpha_{n} + \beta_{i+1})})$

$\oplus  \mathbb{C}_{-(\beta_{n-1} + 2\alpha_{n})}.$ 

By Lemma \ref{lemma 6.2}(1), $H^1(s_{2}\cdots s_{n}s_{1}\cdots s_{n-1}, \alpha_{n-1})=0.$ Therefore by SES, we have

$H^1(v_{k}v_{k+1}s_{n-1}, \alpha_{n-1})=H^1(s_{1}, H^0(s_{2}\cdots s_{n}s_{1}\cdots s_{n-1}, \alpha_{n-1}))=0$ (by Lemma \ref{lemma 2}). \hspace {.8cm} (5.5.0)
By SES and (5.5.0) we have 

$H^1(s_{n}v_{k}v_{k+1}s_{n-1}, \alpha_{n-1})=H^1(s_{n}, H^0(v_{k}v_{k+1}s_{n-1}, \alpha_{n-1}))=\bigoplus\limits_{i=1}^{n-1}\mathbb{C}_{-(\beta_{i}+\alpha_{n})}.$

By Lemma \ref{lemma 2}, we have $H^1(s_{n-1}, H^0(s_{n}v_{k}v_{k+1}s_{n-1}, \alpha_{n-1}))=0.$ Therefore by SES, we have $H^1(s_{n-1}s_{n}v_{k}v_{k+1}s_{n-1}, \alpha_{n-1})=H^0(s_{n-1}, H^1(s_{n}v_{k}v_{k+1}s_{n-1}, \alpha_{n-1})).$

 Since $\langle -(\beta_{n-1}+ \alpha_{n}), \alpha_{n-1} \rangle=-1$ and $\langle -(\beta_{i} + \alpha_{n}), \alpha_{n-1}\rangle =0$ for all $1\le i\le n-2, $ we have $H^1(s_{n-1}s_{n}v_{k}v_{k+1}s_{n-1}, \alpha_{n-1})=\bigoplus\limits_{i=1}^{n-2}\mathbb{C}_{-(\beta_{i} + \alpha_{n})}.$ 
Proceeding in this way
recurssively, we see that $H^1(v_{k-1}v_{k}v_{k+1}s_{n-1}, \alpha_{n-1})=\bigoplus\limits_{i=1}^{a_{k-1}-1}\mathbb{C}_{-(\beta_{i} + \alpha_{n})}.$ Hence Step 1 follows. 

 By Lemma \ref {lemma 5.2}, we have $H^1(s_{n}, H^0(v_{k-1}v_{k}v_{k+1}s_{n-1}, \alpha_{n-1}) )=0.$  \hspace{4.5cm}  (5.5.1)

Therefore by using SES and (5.5.1) we have

$H^1(s_{n}v_{k-1}v_{k}v_{k+1}s_{n-1} , \alpha_{n-1})=H^0(s_{n}, H^1(v_{k-1}v_{k}v_{k+1}s_{n-1}, \alpha_{n-1})).$      \hspace{4cm}    (5.5.2) 

Since $\langle -(\beta_{j} + \alpha_{n}), \alpha_{n}\rangle=0$ for all $1\le j\le a_{k-1}-1,$ by using (5.5.2) we have

$H^1(s_{n}v_{k-1}v_{k}v_{k+1}s_{n-1} , \alpha_{n-1})=\bigoplus\limits_{j=1}^{a_{k-1}-1}\mathbb{C}_{-(\beta_{j} + \alpha_{n})}.$

By Lemma \ref {lemma 2}, we have $H^1(s_{n-1}, H^0(s_{n}v_{k-1}v_{k}v_{k+1}s_{n-1} , \alpha_{n-1}) )=0.$  \hspace{3.9cm}(5.2.3)

Therefore by using SES and (5.2.3) we have

$H^1(s_{n-1}s_{n}v_{k-1}v_{k}v_{k+1}s_{n-1} , \alpha_{n-1})= H^0(s_{n-1}, H^1(s_{n}v_{k-1}v_{k}v_{k+1}s_{n-1}, \alpha_{n-1})).$ \hspace{2.6cm}(5.2.4)

Since $\langle -(\beta_{j} + \alpha_{n}), \alpha_{n-1}\rangle=0$ for all $1\le j\le a_{k-1}-1,$ by using (6.2.4) we have 

$H^1(s_{n-1}s_{n}v_{k-1}v_{k}v_{k+1}s_{n-1}, \alpha_{n-1})=\bigoplus\limits_{j=1}^{a_{k-1}-1}\mathbb{C}_{-(\beta_{j} + \alpha_{n})}.$

Proceeding recursively we have

$$H^1(s_{a_{k-2}}\cdots s_{n-1}s_{n}v_{k-1}v_{k}v_{k+1}s_{n-1} , \alpha_{n-1})=\bigoplus\limits_{j=1}^{a_{k-1}-1}\mathbb{C}_{-(\beta_{j} + \alpha_{n})}.$$

Using the similar arguments as above we have  

$$H^1(u_{1}, \alpha_{n-1})=\bigoplus\limits_{j=1}^{a_{k-1}-1}\mathbb{C}_{-(\beta_{j} + \alpha_{n})}.$$

\end{proof}
\begin{cor}\label{cor 6.6}
If  $H^1(u_{1}, \alpha_{n-1})_{\mu}\neq 0,$ then $H^0(w_{k-1}s_{n}s_{1}s_{2}\cdots s_{n + 1-k},  \alpha_{n+1-k})_{\mu}\neq0.$
\end{cor}
\begin{proof}
	Corollary follows from Lemma \ref {lemma 5.4} and Lemma \ref{lemma 6.5}. 
\end{proof}
\begin{cor}\label{cor 6.7}
 If $H^1(u_{1}, \alpha_{n-1})_{\mu}\neq 0,$ then $H^0(w_{k}s_{n}, \alpha_{n})_{\mu}\neq0.$
\end{cor}
\begin{proof}
	Corollary follows from Lemma \ref{lemma 5.5} and Lemma \ref{lemma 6.5}.	
\end{proof}
Let $3\le r \le k.$ Let $M_{r}:=\{\mu \in X(T): H^1(w_{r}, \alpha_{n-1})_{\mu}\neq 0 \}$  and $M_{0}:=\{\mu\in X(T): H^1(u_{1}, \alpha_{n-1})_{\mu}\neq 0 \}.$
Then we have

\begin{cor}\label{cor 6.8}
	\item [(1)] $M_{r}\cap M_{r'}=\emptyset$ whenever $r\neq r'.$
	\item[(2)] $M_{0}\cap M_{r}=\emptyset$ for every $1\le r \le k.$
\end{cor}
\begin{proof}
	Proof of (1) and (2) follow from Lemma \ref{lemma 6.2} and Lemma \ref{lemma 6.5}. 
\end{proof}

\begin{lem}\label{lemma 3.3}
Let $c=s_{1}s_{2}\cdots s_{n}.$ Let $T_{r}=c^{r-1}s_{1}s_{2}\cdots s_{n-1},$  for all $ 2\le r \le n.$ Then, $T_{r}(\alpha_{j})<0$ for all $n+1-r\le j \le n-1.$
\end{lem}
\begin{proof}
Assume $r=2.$ Then we see that $T_{r}(\alpha_{n-1})=-\alpha_{0}.$ 
We assume that for $2<l<n,$ we have $T_{l}(\alpha_{j})<0$ for $n+1-l \le j \le n-1.$

Note that $T_{l+1}=T_{l}s_{n}s_{1}\cdots s_{n-1}.$ Then for all $n-l\le i\le n-3,$ we have $s_{1}\cdots s_{n-1}(\alpha_{i})=\alpha_{i+1}.$ Since $n-(i+1)\ge 2$ and  $n+1-l\le i+1\le n-1,$ we have  $T_{l+1}(\alpha_{i})=T_{l}(\alpha_{i+1})<0$ (by assumption). Since
$s_{1}\cdots s_{n-1}(\alpha_{n-2})=\alpha_{n-1},$ we have $T_{l+1}(\alpha_{n-2})=T_{l}(s_{n}(-\alpha_{n-1})).$

Since $s_{n}(\alpha_{n-1})=\alpha_{n-1}+ 2\alpha_{n},$ we have $T_{l}s_{n}(\alpha_{n-1})=T_{l}(\alpha_{n-1} + 2\alpha_{n}).$  As $s_{1}\cdots s_{n-1}(\alpha_{n-1} + 2\alpha_{n})=\beta_{1} + 2\alpha_{n},$ we have 
$T_{l}(\alpha_{n-1}+ 2\alpha_{n})=T_{l-1}s_{n}(\beta_{1} + 2\alpha_{n})$.

Since $s_{n}s_{1}\cdots s_{n-1}s_{n}(\beta_{1} + 2\alpha_{n})=-\alpha_{1},$ we have  $T_{l-1}s_{n}(\beta_{1} + 2\alpha_{n})=T_{l-2}(-\alpha_{1}).$ 
Since $s_{n}s_{1}\cdots s_{n-1}(-\alpha_{1})=-\alpha_{2},$ we have $T_{l-2}(-\alpha_{1})=T_{l-3}(-\alpha_{2}).$    Therefore by recursively we have
$T_{l-3}(-\alpha_{2})=-\alpha_{l-1}.$ Hence $T_{l+1}(\alpha_{n-2})=-\alpha_{l-1}<0.$ 
Also it is clear that $T_{l+1}(\alpha_{n-1})<0.$ 
Therefore we have 
$T_{l+1}(\alpha_{j})<0$ for all $n-l \le j\le n-1.$ Hence the result follows.
\end{proof}

\begin{lem}\label{lemma 6.10}
	Let $u,v\in W,$ let $v:=(\prod\limits_{j=1}^{n}s_{j})^{l-1}s_{1}s_{2}\cdots s_{n-1}$ for some positive integer $l\le n,$ such that $l(uv)=l(u) + l(v).$ Let $w=uv.$ If $l\ge 3,$ then $H^i(w,\alpha_{n-1})=0$ for all $i\ge0.$ 
\end{lem}
\begin{proof}
	We note that by SES, we have
	  $$H^0(w, \alpha_{n-1})=H^0(u,H^0(v , \alpha_{n-1})).$$
	  
We show that  $H^0(v , \alpha_{n-1})=0.$

By Lemma \ref{lemma 3.3} we have for each $1\le r\le n-1,$ $c^{r-1}s_{1}\cdots s_{n-1}(\alpha_{j})<0$ for all $n+1-r\le j\le n-1 .$ In particular, we have $l(vs_{n-2})=l(v)-1.$ 

Therefore, by Lemma \ref{lemma 1.2}(4) and using SES, we have $$H^0(v, \alpha_{n-1})=H^0(vs_{n-2}, H^0(s_{n-2},\alpha_{n-1}))=0.\hspace{6.3cm} (5.10.1)$$

Next we show that $H^1(w, \alpha_{n-1})=H^0(u, H^1(v, \alpha_{n-1})).$

We will prove by induction $l(u).$ If $l(u)=0$ then it follows trivially. Next suppose that $l(u)>1.$ Then there exists a simple root $\gamma$ such that $l(s_{\gamma}u)=l(u)-1.$ By using SES and (5.10.1) we have $H^0(s_{\gamma}uv, \alpha_{n-1})=0.$ Again, by induction hypotheses $H^1(s_{\gamma}uv, \alpha_{n-1})=H^0(s_{\gamma}u, H^1(v, \alpha_{n-1})).$
Therefore by SES, we have
$$ H^1(w,  \alpha_{n-1})=H^0(s_{\gamma}, H^1(s_{\gamma}uv, \alpha_{n-1}))=H^0(s_{\gamma}, H^0(s_{\gamma}u, H^1(v, \alpha_{n-1})))=H^0(u, H^1(v, \alpha_{n-1})).$$

$H^1(v, \alpha_{n-1})=0,$ follows from the fact that $l(vs_{n-2})=l(v)-1$ and Lemma \ref{lemma 1.2}(4). Thus, we have $H^1(w, \alpha_{n-1})=0.$ Therefore by \cite[Corollary 6.4, p.780]{Ka} we have $H^i(w, \alpha_{n-1})=0$ for all $i\ge 0.$

\end{proof}

\section{cohomology modules of the tangent bundle of $Z(w, \underline{i})$}
Let $w\in W$ and let $w=s_{i_{1}}s_{i_{2}}\cdots s_{i_{r}}$ be a reduced expression for $w$ and let $\underline{i}=(i_{1},i_{2},\dots ,i_{r})$. Let $\tau=s_{i_{1}}s_{i_{2}}\cdots s_{i_{r-1}}$ and $\underline{i'}=(i_{1},i_{2},\dots, i_{r-1}).$

Recall the following long exact sequence of $B$-modules from \cite{CKP} (see \cite[Proposition 3.1, p.673]{CKP}):

$$0\rightarrow H^0(w,\alpha_{i_{r}})\rightarrow H^0(Z(w,\underline{i}), T_{(w,\underline{i})})\rightarrow H^0(Z(\tau,\underline{i'}), T_{(\tau,\underline{i'})}) \rightarrow$$ $$H^1(w,\alpha_{i_{r}})\rightarrow H^1(Z(w,\underline{i}), T_{(w,\underline{i})})\rightarrow H^1(Z(\tau,\underline{i'}), T_{(\tau,\underline{i'})})\rightarrow H^2(w,\alpha_{i_{r}})\rightarrow $$
$$H^2(Z(w,\underline{i}), T_{(w,\underline{i})})\rightarrow H^2(Z(\tau,\underline{i'}), T_{(\tau,\underline{i'})})\rightarrow H^3(w,\alpha_{i_{r}})\rightarrow \cdots$$

By \cite[Corollary 6.4, p.780]{Ka}, we have $H^j(w, \alpha_{i_{r}})=0$ for every $j\ge 2.$ Thus we have the following exact sequence of $B$-modules:

$$0\rightarrow H^0(w,\alpha_{i_{r}})\rightarrow H^0(Z(w,\underline{i}), T_{(w,\underline{i})})\rightarrow H^0(Z(\tau,\underline{i'}), T_{(\tau,\underline{i'})}) \rightarrow$$ $$H^1(w,\alpha_{i_{r}})\rightarrow H^1(Z(w,\underline{i}), T_{(w,\underline{i})})\rightarrow H^1(Z(\tau,\underline{i'}), T_{(\tau,\underline{i'})})\rightarrow 0 $$

Now onwards we call this exact sequence by LES.

Let $w_{0}=s_{j_{1}}s_{j_{2}}\cdots s_{j_{N}}$ be a reduced expression of $w_{0}$ such that $\underline{i}=(j_{1},j_{2},\dots, j_{r})$ and let and $\underline{j}=(j_{1},j_{2},\dots,j_{N}).$

\begin{lem}\label{lemma 7.1}
The natural homomorphism 
$$ f:H^1(Z(w_{0}, \underline{j}), T_{(w_{0}, \underline{j})})\longrightarrow H^1(Z(w, \underline{i}), T_{(w, \underline{i})})$$
of $B$- modules is surjective.
\end{lem}
\begin{proof}
(see \cite[Lemma 7.1, p.459]{CK}).	
\end{proof}

\begin{lem}\label{lemma 7.2}
Let $J=S\setminus \{\alpha_{n-1}\}.$ Let $v\in W_{J}$ and $u\in W$ be such that $l(uv)=l(u)+l(v).$ Let $u=s_{i_{1}}\cdots s_{i_{r}}$ and $v=s_{i_{r+1}}\cdots s_{i_{t}}$ be reduced expressions of $u$ and $v$ respectively. Let $\underline{i}=(i_{1},i_{2},\dots,i_{r})$ and $\underline{j}=(i_{1},i_{2},\dots,i_{r},i_{r+1},\dots , i_{t}).$
Then we have 
\begin{enumerate}
	\item [(1)] The natural homomorphism 
	$$H^0(Z(uv,\underline{j}), T_{(uv,\underline{j})})\longrightarrow H^0(Z(u, \underline{i}), T_{(u, \underline{i})})$$ 
	
of $B$-modules is surjective. 

\item[(2)] The natural homomorphism
 $$H^1(Z(uv, \underline{j}), T_{(uv, \underline{j})})\longrightarrow H^1(Z(u, \underline{i}), T_{(u, \underline{i})})$$ 
 
of $B$-modules is an isomorphism.
\end{enumerate}

\end{lem}
\begin{proof}
	Let $r+1\le l\le t.$ Let $v_{l}=us_{i_{r+1}}\cdots s_{i_{l}}$ and $\underline{i}_{l}=(\underline{i},i_{r+1},\dots, i_{l}).$

Proof of (1): By Lemma \ref{lemma 2.5} we have $H^1(v_{t}, \alpha_{i_{t}})=0.$ Therefore, using LES the natural homomorphism
\begin{center}
	$H^0(Z(v_{t},\underline{i}_{t}), T_{(v_{t}, \underline{i}_{t})})\longrightarrow H^0(Z(v_{t-1},\underline{i}_{t-1}), T_{(v_{t-1}, \underline{i}_{t-1})})$ 
\end{center}
is surjective.

 By induction on $l(v)$, the natural homomorphism 
\begin{center}

$H^0(Z(v_{t-1},\underline{i}_{t-1}), T_{(v_{t-1}, \underline{i}_{t-1})})\longrightarrow H^0(Z(u,\underline{i}), T_{(u, \underline{i})})$ 

\end{center}
is surjective. Hence we conclude that the natural homomorphism

\begin{center}
	
	$H^0(Z(uv,\underline{j}), T_{(uv, \underline{j})})\longrightarrow H^0(Z(u,\underline{i}), T_{(u, \underline{i})})$ 
	
\end{center}
is surjective.

Proof of (2): We will prove by induction on $l(v).$ By LES, we have the following exact sequence of $B$-modules:
\begin{center}
$	0\longrightarrow H^0(uv, \alpha_{i_{t}}) \longrightarrow H^0(Z(uv,\underline{j}), T_{(uv, \underline{_{j}})})\longrightarrow
H^0(Z(v_{t-1},\underline{i}_{t-1}), T_{(v_{t-1},\underline{i}_{t-1})})\longrightarrow H^1(uv, \alpha_{i_{t}})\longrightarrow H^1(Z(uv, \underline{j}),  T_{(uv,\underline{j})})\longrightarrow H^1(Z(v_{t-1}, \underline{i}_{t-1}), T_{(v_{t-1}, \underline{i}_{t-1})})\longrightarrow 0. $
\end{center} 
By induction on $l(v),$ the natural homomorphism 

\begin{center}
	
	$H^1(Z(v_{t-1},\underline{i}_{t-1}), T_{(v_{t-1}, \underline{i}_{t-1})})\longrightarrow H^1(Z(u,\underline{i}), T_{(u, \underline{i})})$ 
	
\end{center}
is an isomorphism.

By Lemma \ref {lemma 2.5}, $H^1(uv, \alpha_{i_{t}})=0.$ Therefore, by the above exact sequence we have  $ H^1(Z(uv, \underline{j}),  T_{(uv,\underline{j})})\longrightarrow H^1(Z(v_{t-1}, \underline{i}_{t-1}), T_{(v_{t-1}, \underline{i}_{t-1})})$ is an isomorphism. Hence we conclude that the homomorphism   

\begin{center}
	$H^1(Z(uv,\underline{j}), T_{(uv, \underline{j})})\longrightarrow H^1(Z(u,\underline{i}), T_{(u, \underline{i})})$ 
\end{center} 
of $B$-modules is an isomorphism.
\end{proof}

Recall that by Lemma \ref {lemma 3.2}(1)  and Lemma \ref {lemma 3.4}(2) we have 
\begin{center}
	$w_{0}=(\prod\limits_{l_{1}=1}^{k-1}[a_{l_{1}}, n])([a_{k}, n]^{n+1-k})(\prod\limits_{l_{2}=1}^{k-1}[a_{k}, a_{l_{2}}-1])$
\end{center} 
is a reduced expression for $w_{0}.$ Let $\underline{i}$ be the tuple corresponding to this reduced of $w_{0}.$ Let $u_{1}=w_{k}s_{n}[a_{k}, n-1]$ and $\underline{i}_{1}$ be the tuple corresponding to the reduced expression $\prod\limits_{l_{1}=1}^{k}[a_{l_{1}}, n]([a_{k}, n-1]).$ Note that $a_{k}=1.$ With this notation, we have

\begin{lem}\label{lemma 7.3}
\begin{enumerate}
\item[(1)] The natural homomorphism 
$$H^0(Z(w_{0}, \underline{i}), T_{(w_{0}, \underline{i})})\longrightarrow H^0(Z(u_{1},\underline{i}_{1}), T_{(u_{1}, \underline{i}_{1})})$$ 
of $B$-modules is an isomorphism.

\item[(2)] The natural homomorphism 
$$H^1(Z(w_{0}, \underline{i}), T_{(w_{0}, \underline{i})})\longrightarrow H^1(Z(u_{1},\underline{i}_{1}), T_{(u_{1}, \underline{i}_{1})})$$ 
of $B$-modules is an isomorphism.	

\end{enumerate}
\begin{proof}
For $1\le n-k$, let $u_{j}=w_{k}s_{n}[a_{k}, n]^{j-1}s_{1}s_{2}\cdots s_{n-1}$ and $\underline{i}_{j}$ be the tuple corresponding to the reduced expression $u_{j}=(\prod\limits_{l_{1}=1}^{k}[a_{l_{1}}, n])([a_{k}, n])^{j-1}[a_{k}, n-1]$ (see Lemma \ref {lemma 3.4}(2)). Note that 
$w_{0}=u_{n-k}s_{n}(\prod\limits_{l_{2}=1}^{k-1}[a_{k}, a_{l_{2}}-1]).$  

{\bf Case 1: $a_{1}\neq n$}. In this case, we have  $s_{n}(\prod\limits_{l_{2}=1}^{k-1}[a_{k}, a_{l_{2}-1} ])\in W_{J},$ where $J=S\setminus\{\alpha_{n-1}\}.$

Then by Lemma \ref {lemma 7.2}, the natural homomorphism 
$$H^0(Z(w_{0}, \underline{i}), T_{(w_{0}, \underline{i})})\longrightarrow H^0(Z(u_{n-k},\underline{i}_{n-k}), T_{(u_{n-k}, \underline{i}_{n-k})})\hspace{5cm}        (6.3.1)$$
is surjectie and the natural homomorphism 
$$H^1(Z(w_{0}, \underline{i}), T_{(w_{0}, \underline{i})})\longrightarrow H^1(Z(u_{n-k},\underline{i}_{n-k}), T_{(u_{n-k}, \underline{i}_{n-k})})$$ of $B$-modules is an isomorphism. \hspace{8.8cm}	 (6.3.2)
		
If $j\ge 2,$ then by Lemma	\ref {lemma 6.10}, we have $H^1(u_{j}, \alpha_{n-1})=0.$ Let $u'_{j}=u_{j}s_{n-1}$ and let $\underline{i'}_{j}$ be the partial subsequence of $u'_{j}$ such that $\underline{i}_{j}=(\underline{i'}_{j}, n-1).$ Hence by LES , we observe that the natural homomorphism
$$H^0(Z(u_{j}, \underline{i}_{j}), T_{(u_{j}, \underline{i}_{j})})\longrightarrow H^0(Z(u'_{j},\underline{i'}_{j}), T_{(u'_{j}, \underline{i'}_{j})})\hspace{6.5cm}	      (6.3.3)$$ is surjective and $$ H^1(Z(u_{j}, \underline{i}_{j}), T_{(u_{j}, \underline{i}_{j})})\longrightarrow H^1(Z(u'_{j},\underline{i'}_{j}), T_{(u'_{j}, \underline{i'}_{j})})\hspace{6.5cm}       (6.3.4)$$ is an isomorphism.
Therefore by Lemma \ref{lemma 7.2} we have  the natural homomorphism $$H^0(Z(u'_{j}, \underline{i'}_{j}), T_{(u'_{j}, \underline{i'}_{j})})\longrightarrow H^0(Z(u_{j-1},\underline{i}_{j-1}), T_{(u_{j-1}, \underline{i}_{j-1})})\hspace{5.1cm}	         (6.3.5)$$
is surjective and 
$$ H^1(Z(u'_{j}, \underline{i'}_{j}), T_{(u'_{j}, \underline{i'}_{j})})\longrightarrow H^1(Z(u_{j-1},\underline{i}_{j-1}), T_{(u_{j-1}, \underline{i}_{j-1})})\hspace{5cm}               (6.3.6)$$ is an isomorphism. Therefore, by combining  (6.3.3), (6.3.5) together and (6.3.4), (6.3.6) together we see that the homomorphism 
$$ H^0(Z(u_{j}, \underline{i}_{j}), T_{(u_{j}, \underline{i}_{j})})\longrightarrow H^0(Z(u_{j-1},\underline{i}_{j-1}), T_{(u_{j-1}, \underline{i}_{j-1})})$$                
is surjective and $$ H^1(Z(u_{j}, \underline{i}_{j}), T_{(u_{j}, \underline{i}_{j})})\longrightarrow H^1(Z(u_{j-1},\underline{i}_{j-1}), T_{(u_{j-1}, \underline{i}_{j-1})})$$ is an isomorphism.

Proceeding recursively  we get that the homomorphism 
$$H^0(Z(w_{0}, \underline{i}), T_{(w_{0}, \underline{i})})\longrightarrow H^0(Z(u_{1},\underline{i}_{1}), T_{(u_{1}, \underline{i}_{1})})$$ 
of $B$-modules is surjective and homomorphism 
$$H^1(Z(w_{0}, \underline{i}), T_{(w_{0}, \underline{i})})\longrightarrow H^1(Z(u_{1},\underline{i}_{1}), T_{(u_{1}, \underline{i}_{1})})$$ of $B$-modules is an isomorphism.

Further since  $u_{1}^{-1}(\alpha_{0})<0,$ by \cite[Lemma 6.2, p.667]{CKP}, we have $H^0(Z(u_{1}, \underline{i}_{1}), T_{(u_{1}, \underline{i}_{1})})_{-\alpha_{0}}\neq 0.$ By \cite[Theorem 7.1]{CKP}, $H^0(Z(w_{0}, \underline{i}), T_{(w_{0}, \underline{i})})$ is a parabolic subalgebra of $\mathfrak{g}$ and hence there is a unique $B$-stable line in $H^0(Z(w_{0}, \underline{i}), T_{(w_{0}, \underline{i})}),$ namely $\mathfrak{g}_{-\alpha_{0}}.$ Thereore we conclude that the natural homomorphism
$$H^0(Z(w_{0}, \underline{i}), T_{(w_{0}, \underline{i})})\longrightarrow H^0(Z(u_{1},\underline{i}_{1}), T_{(u_{1}, \underline{i}_{1})})$$ 
of $B$-modules is an isomorphism.

{\bf Case 2 : $a_{1}=n$}

Then by Lemma \ref {lemma 7.2}, the natural morphism

$$H^0(Z(w_{0}, \underline{i}), T_{(w_{0}, \underline{i})})\longrightarrow H^0(Z(u_{n+1-k},\underline{i}_{n+1-k}), T_{(u_{n+1-k}, \underline{i}_{n+1-k})})$$ 

is surjectie and the natural homomorphism 

$$H^1(Z(w_{0}, \underline{i}), T_{(w_{0}, \underline{i})})\longrightarrow H^1(Z(u_{n+1-k},\underline{i}_{n+1-k}), T_{(u_{n+1-k}, \underline{i}_{n+1-k})})$$ 
of $B$-modules is an isomorphism.	

If $j\ge 2,$ then by Lemma	\ref {lemma 6.10}, we have $H^1(u_{j}, \alpha_{n-1})=0.$ Hence by LES, for each $2\le j\le n+1-k$, we observe that the natural homomorphism
$$H^0(Z(u_{j}, \underline{i}_{j}), T_{(u_{j}, \underline{i}_{j})})\longrightarrow H^0(Z(u_{j-1},\underline{i}_{j-1}), T_{(u_{j-1}, \underline{i}_{j-1})})$$	
is surjective and 
$$ H^1(Z(u_{j}, \underline{i}_{j}), T_{(u_{j}, \underline{i}_{j})})\longrightarrow H^1(Z(u_{j-1},\underline{i}_{j-1}), T_{(u_{j-1}, \underline{i}_{j-1})})$$
is an isomorphism. Therefore, the homomorphism 	
$$H^0(Z(w_{0}, \underline{i}), T_{(w_{0}, \underline{i})})\longrightarrow H^0(Z(u_{1},\underline{i}_{1}), T_{(u_{1}, \underline{i}_{1})})$$ of $B$-modules is surjective and the homomorphism 
$$H^1(Z(w_{0}, \underline{i}), T_{(w_{0}, \underline{i})})\longrightarrow H^1(Z(u_{1},\underline{i}_{1}), T_{(u_{1}, \underline{i}_{1})})$$ of $B$-modules is an isomorphism.

Further since  $u_{1}^{-1}(\alpha_{0})<0,$ by \cite[Lemma 6.2, p.667]{CKP}, we have $H^0(Z(u_{1}, \underline{i}_{1}), T_{(u_{1}, \underline{i}_{1})})_{-\alpha_{0}}\neq 0.$ By \cite[Theorem 7.1]{CKP}, $H^0(Z(w_{0}, \underline{i}), T_{(w_{0}, \underline{i})})$ is a parabolic subalgebra of $\mathfrak{g}$ and hence there is a unique $B$-stable line in $H^0(Z(w_{0}, \underline{i}), T_{(w_{0}, \underline{i})}),$ namely $\mathfrak{g}_{-\alpha_{0}}.$ Thereore we conclude that the natural homomorphism
$$H^0(Z(w_{0}, \underline{i}), T_{(w_{0}, \underline{i})})\longrightarrow H^0(Z(u_{1},\underline{i}_{1}), T_{(u_{1}, \underline{i}_{1})})$$ 
of $B$-modules is an isomorphism.
\end{proof}

The following is a useful Corollary.
\begin{cor} \label{cor 7.4}
If $\mu \in X(T)\setminus \{0\},$ then, we have $$dim(H^0(Z(u_{1}, \underline{i}_{1}), T_{(u_{1}, \underline{i}_{1})} )_{\mu})\le 1.$$

\begin{proof}	
By \cite[Theorem 7.1]{CKP}, $H^0(Z(w_{0}, \underline{i}), T_{(w_{0}, \underline{i})})$ is a parabolic subalgebra of $\mathfrak{g}.$ By Lemma \ref{lemma 7.3}(1) we have

$$H^0(Z(w_{0}, \underline{i}), T_{(w_{0}, \underline{i})})\simeq H^0(Z(u_{1}, \underline{i_{1}}), T_{(u_{1}, \underline{i_{1}})}).$$ (as $B$-modules).

Hence for any $\mu \in X(T)\setminus \{0\},$ we have 

$$dim(H^0(Z(u_{1}, \underline{i}_{1}), T_{(u_{1}, \underline{i}_{1})})_{\mu})\le 1.$$
\end{proof}
\end{cor}

Let $u'_{1}=w_{k}s_{n}[a_{k}, n-2]$ and let $\underline{i'}_{1}$ be the tuple corresponding to the reduced expression $\prod\limits_{l_{1}=1}^{k}[a_{l_{1}}, n][a_{k}, n-2].$ Let $3\le r \le k,$ and let ${\underline{j}}_{r}=(a_{1},\dots,n;a_{2},\dots,n;\dots; a_{r-1},\dots, n; a_{r},\dots,n-1)$ and $\underline{j'}_{r}=(a_{1},\dots,n;a_{2},\dots,n; a_{r},\dots,n-2).$ 

We now prove 

\begin{lem}\label{lemma 7.5}
Let $\mu \in X(T)\setminus \{0\}.$ 

\begin{enumerate}
	\item [(1)] If $H^1(u_{1}, \alpha_{n-1})_{\mu}=0,$ then dim($H^0(Z(u'_{1}, \underline{i'}_{1}), T_{(u'_{1}, \underline{i'}_{1})})_{\mu}$)$\le 1.$
	\item[(2)] If $H^1(u_{1}, \alpha_{n-1})_{\mu}\neq0,$ then dim($H^0(Z(u'_{1}, \underline{i'}_{1}), T_{(u'_{1}, \underline{i'}_{1})})_{\mu}$)$= 2,$ and the natural homomorphism
	
	$$H^0(Z(u'_{1}, \underline{i'}_{1}), T_{(u'_{1}, \underline{i'}_{1})})_{\mu}\longrightarrow H^1(u_{1}, \alpha_{n-1})_{\mu}$$ is surjective.

\end{enumerate}
\end{lem}
\begin{proof}
	By LES, we have the following long exact sequence of $B$-modules:
$$ 0\longrightarrow H^0(u_{1}, \alpha_{n-1})\longrightarrow H^0(Z(u_{1}, \underline{i}_{1}), T_{(u_{1}, \underline{i_{1}})})\longrightarrow $$
$$H^0(Z(u'_{1}, \underline{i'}_{1}), T_{(u'_{1}, \underline{i'}_{1})})\longrightarrow H^1(u_{1}, \alpha_{n-1})\longrightarrow \cdots \hspace{5.7cm}            (6.5.1)$$

Proof of (1): Assume that $H^1(u_{1}, \alpha_{n-1})_{\mu}=0$ and $\mu \in X(T)\setminus \{0\},$ then by the above exact sequence the natural homoorphism (of $T$-modules) $H^0(Z(u_{1}, \underline{i}_{1}), T_{(u'_{1}, \underline{i}_{1})})_{\mu}\longrightarrow H^0(Z(u'_{1}, \underline{i'}_{1}), T_{(u_{1}', \underline{i'}_{1})})_{\mu} $ is surjective. By Corollary \ref{cor 7.4}, we have dim($H^0(Z(u_{1}, \underline{i}_{1}), T_{(u_{1},\underline{i}_{1})})_{\mu}$)$\le 1.$

Proof of (2): Assume that $H^1(u_{1}, \alpha_{n-1})_{\mu}\neq0.$ Then by Lemma \ref{lemma 6.5}, $\mu$ is of the form $-(\beta_{j}+\alpha_{n})$ for some $1 \le  j \le a_{k-1}-1$ and dim($H^1(u_{1}, \alpha_{n-1})_{\mu}$)=1. Hence by using Corollary \ref{cor 7.4}, we see that if $H^1(u_{1}, \alpha_{n-1})_{\mu}\neq 0,$ then 

$dim(H^0(Z(u'_{1}, \underline{i'}_{1}), T_{(u'_{1}, \underline{i'}_{1})})_{\mu})\le 2.$ \hspace{8cm}   (6.5.2)
 
Let $\underline{j''}_{k}=(\underline{j}_{k}, n)$ be the tuple corresponding to the reduced expression $w_{k}s_{n}=\prod\limits_{l_{1}=1}^{k}[a_{l_{1}}, n].$ Then
by Lemma \ref{lemma 7.2} the natural homomorphism 
\begin{center}
	$H^0(Z(u'_{1}, \underline{i'}_{1}), T_{(u'_{1}, \underline{i'}_{1})})\longrightarrow H^0(Z(w_{k}s_{n}, \underline{j''}_{k}), T_{(w_{k}s_{n}, \underline{j''}_{k})})$  
\end{center}
is surjective. By (6.5.2), we have dim($H^0(Z(w_{k}s_{n}, \underline{j''}_{k}), T_{(w_{k}s_{n}, \underline{j''}_{k})})_{\mu}$)$\le2.$ \hspace{1.5cm}(6.5.3) 

Again by the Lemma \ref{lemma 7.2} the natural homomorphism
\begin{center}
	$H^0(Z(w_{k}s_{n}, \underline{j''}_{k}), T_{(w_{k}s_{n}, \underline{j''}_{k})})\longrightarrow H^0(Z(w_{k}, \underline{j}_{k}), T_{(w_{k}, \underline{j}_{k})})$ \hspace{4.5cm}(6.5.4)  
\end{center}
is surjective. Recall that $u'_{t}=u_{k}s_{n-1}.$ By LES we have the following long exact sequence of $B$-modules:
\begin{center}
	$0\longrightarrow H^0(w_{k}, \alpha_{n-1})\longrightarrow H^0(Z(w_{k}, \underline{j}_{k}), T_{(w_{k}, \underline{j}_{k})})\longrightarrow $

	$H^0(Z(u'_{k}, \underline{i'}_{k}), T_{(u'_{k}, \underline{i'}_{k})})\longrightarrow H^1(w_{k}, \alpha_{n-1})\longrightarrow \cdots$
\end{center}
Since $H^1(u_{1}, \alpha_{n-1})_{\mu}\neq0,$ by Corollary \ref {cor 6.8} we have $H^1(w_{k}, \alpha_{n-1})_{\mu}=0.$
Therefore we have an exact sequence 
\begin{center}
	$0\longrightarrow H^0(w_{k}, \alpha_{n-1})_{\mu}\longrightarrow H^0(Z(w_{k}, \underline{j}_{k}), T_{(w_{k}, \underline{j}_{k})})_{\mu}\longrightarrow $

$H^0(Z(\tau_{k}, \underline{j'}_{k}), T_{(\tau_{k}, \underline{j'}_{k})})_{\mu}\longrightarrow 0.$
\end{center}
Let $\sigma_{k}=w_{k-1}s_{n}s_{1}\cdots s_{n+1-k}$ and $\underline{j^*}_{k}=(\underline{j}_{k-1}, n, 1, 2,\dots , n+1-k)$ be the tuple corresponding to this reduced expression of $\sigma_{k}.$ Then by using Lemma \ref {lemma 7.2}, the natural homomorphism 
\begin{center}
	$H^0(Z(\tau_{k}, \underline{j'}_{k}), T_{(\tau_{k}, \underline{j'}_{k})})\longrightarrow H^0(Z(\sigma_{k}, \underline{j^*}_{k}), T_{(\sigma_{k},  \underline{j^*}_{k})})$     \hspace{5cm}(6.5.5)
\end{center} 
is surjective. Therefore the natural map
$$H^0(Z(w_{k}, \underline{j}_{k}), T_{(w_{k}, \underline{j}_{k})})_{\mu}\longrightarrow H^0(Z(\sigma_{k}, \underline{j^*}_{k}),   T_{(\sigma_{k}, \underline{j^*}_{k})})_{\mu}\hspace{5cm} (6.5.6)$$
is surjective.

Since $H^1(u_{1}, \alpha_{n-1})_{\mu}\neq0,$ by Corollary \ref{cor 6.6} we have $H^0(\sigma_{k},  \alpha_{n+1-k})_{\mu}\neq0.$  Therefore, $H^0(Z(\sigma_{k}, \underline{j^*}_{k}),   T_{(\sigma_{k}, \underline{j^*}_{k})})_{\mu}\neq0.$ Thus from (6.5.6) we have $H^0(Z(w_{k}, \underline{j}_{k}), T_{(w_{k}, \underline{j}_{k})})_{\mu}\neq0.$ Then we have an exact sequence of $T$-modules
\begin{center}
	$0\longrightarrow H^0(w_{k}s_{n}, \alpha_{n})_{\mu}\longrightarrow H^0(Z(w_{k}s_{n}, \underline{j''}_{k}), T_{(w_{k}s_{n}, \underline{j''}_{k})})_{\mu}\longrightarrow H^0(Z(w_{k}, \underline{j}_{k}),   T_{(w_{k}, \underline{j}_{k})})_{\mu}\longrightarrow0$
\end{center}
Since $H^1(u_{1}, \alpha_{n-1})_{\mu}\neq 0,$ by Corollary \ref {cor 6.7} we have $H^0(w_{k}s_{n}, \alpha_{n})_{\mu}\neq0.$

Therefore dim($H^0(Z(w_{k}s_{n}, \underline{j''}_{k}), T_{(w_{k}s_{n},
\underline{j''}_{k})})_{\mu}$)$\ge 2.$ On the other hand by (6.5.3) we have dim($H^0(Z(w_{k}s_{n}, \underline{j''}_{k}), T_{(w_{k}s_{n}, \underline{j''}_{k})})_{\mu}$)$\le2.$ Hence we have dim($H^0(Z(w_{k}s_{n}, \underline{j''}_{k}), T_{(w_{k}s_{n}, \underline{j''}_{k})})_{\mu}$)$=2.$ 

Since the natural homomorphism 
\begin{center}
	$H^0(Z(u'_{1}, \underline{i'}_{1}), T_{(u'_{1}, \underline{i'}_{1})})\longrightarrow H^0(Z(w_{k}s_{n}, \underline{j''}_{k}), T_{(w_{k}s_{n}, \underline{j''}_{k})})$ 
\end{center}
is surjective, we have dim($H^0(Z(u'_{1}, \underline{i'}_{1}), T_{(u'_{1}, \underline{i'}_{1})_{\mu}})$)$=2.$      \hspace{5.4cm} (6.5.7)

It is clear from (6.5.1),(6.5.7), and Corollary \ref{cor 7.4}, that

\begin{center}
	$H^0(Z(u'_{1}, \underline{i'}_{1}), T_{(u'_{1}, \underline{i'}_{1})})_{\mu}\longrightarrow H^1(u_{1}, \alpha_{n-1})_{\mu}$
\end{center}

is surjective.

\end{proof}

\begin{cor}\label{cor 7.6}
	The natural homomorphism 
\begin{center}
	$H^0(Z(u'_{1}, \underline{i'}_{1}), T_{(u'_{1}, \underline{i'}_{1})})\longrightarrow H^1(u_{1}, \alpha_{n-1})$
\end{center}
is surjective.
\end{cor}
\begin{proof}
	Note that by Lemma \ref{lemma 6.5}, if $H^1(u_{1}, \alpha_{n-1})_{\mu}\neq0$ then $\mu \in X(T)\setminus \{0\}.$ Now the proof follows from Lemma \ref {lemma 7.5}.
\end{proof}

\begin{lem}\label{lemma 7.7}

\begin{enumerate}
\item[(1)]

 If $H^1(w_{m}, \alpha_{n-1})_{\mu}=0$ for all $r\le m \le k$ and $H^1(u_{1}, \alpha_{n-1})_{\mu}=0,$ then  dim$(H^0(Z(\tau_{r}, \underline{j'}_{r}), T_{(\tau_{r}, \underline{j'}_{r})})_{\mu})$ $\le 1.$ 

\item[(2)] Let $3\le r\le k.$ If $H^1(w_{r}, \alpha_{n-1})_{\mu}\neq0,$ then dim$(H^0(Z(\tau_{r}, \underline{j'}_{r}), T_{(\tau_{r}, \underline{j'}_{r})})_{\mu})=2,$ and the natural homomorphism

\begin{center}
	$H^0(Z(\tau_{r}, \underline{j'}_{r}), T_{(\tau_{r}, \underline{j'}_{r})})_{\mu}\longrightarrow H^1(w_{r}, \alpha_{n-1})_{\mu}$
\end{center}
is surjective.
\end{enumerate}	
	
\end{lem}
\begin{proof}
Proof of (1): If $H^1(u_{1}, \alpha_{n-1})_{\mu}=0,$ then by Lemma \ref {lemma 7.5}, we have 

dim($H^0(Z(u'_{1}, \underline{i'}_{1}), T_{(u'_{1}, \underline{i'}_{1})})_{\mu}$)$\le 1.$ By Lemma  \ref {lemma 7.2}, the natural homomorphism 
\begin{center}
	$H^0(Z(u'_{1}, \underline{i'}_{1}), T_{(u'_{1}, \underline{i'}_{1})})_{\mu}\longrightarrow H^0(Z(w_{k}, \underline{j}_{k}), T_{(w_{k}, \underline{j}_k)})_{\mu} $
\end{center}
is surjective. If $H^1(w_{m}, \alpha_{n-1})_{\mu}=0$ for all $r\le m\le k,$ by using LES, we see that the natural homomorphism 
\begin{center}
	$H^0(Z(w_{k}, \underline{j}_{k}), T_{(w_{k}, \underline{j}_{k})})_{\mu}\longrightarrow H^0(Z(\tau_{r}, \underline{j'}_{r}), T_{(\tau_{r}, \underline{j'}_{r})})_{\mu} $
\end{center} 
is surjective. Therefore, we have dim($H^0(Z(\tau_{r}, \underline{j'}_{r}), T_{(\tau_{r}, \underline{j'}_{r})})_{\mu}$)$\le 1.$

Proof of (2): Assume that $H^1(w_{r}, \alpha_{n-1})_{\mu}\neq 0.$ Then by Corollary \ref {cor 6.8}, we have $H^1(w_{m}, \alpha_{n-1})_{\mu}=0$ for all $r+1\le m \le k$ and $H^1(u_{1}, \alpha_{n-1})_{\mu}=0.$ Then by $(1),$ we have 
\begin{center}
dim($H^0(Z(\tau_{r+1}, \underline{j'}_{r+1}), T_{(\tau_{r+1}, \underline{j'}_{r+1})})_{\mu}$)$\le 1.$	
\end{center}
By Lemma \ref {lemma 7.2}, the natural homomorphism
\begin{center}
	$H^0(Z(\tau_{r+1}, \underline{j'}_{r+1}), T_{(\tau_{r+1}, \underline{j'}_{r+1})})_{\mu}\longrightarrow H^0(Z(w_{r}, \underline{j}_{r}), T_{(w_{r}, \underline{j}_{r})})_{\mu}$
\end{center}
is surjective. Hence, we have dim($H^0(Z(w_{r}, \underline{j}_{r}), T_{(w_{r}, \underline{j}_{r})})_{\mu}$)$\le 1.$ \hspace{5cm} (6.7.1)

By LES we have the following long exact sequence of $B$-modules:

\begin{center}
	$0\longrightarrow H^0(w_{r}, \alpha_{n-1})\longrightarrow H^0(Z(w_{r}, \underline{j}_{r}), T_{(w_{r}, \underline{j}_{r})})\longrightarrow $

$H^0(Z(\tau_{r}, \underline{j'}_{r}), T_{(\tau_{r}, \underline{j'}_{r})})\longrightarrow H^1(w_{r}, \alpha_{n-1})\longrightarrow \cdots$
\end{center}
Since dim($H^0(Z(w_{r}, \underline{j}_{r}), T_{(w_{r}, \underline{j}_{r})})_{\mu}$)$\le 1$ and dim($H^1(w_{r}, \alpha_{n-1})_{\mu}$)$=1,$ 
we see that 

dim($H^0(Z(\tau_{r}, \underline{j'}_{r}))_{\mu}$)$\le 2.$  \hspace{10.9cm}   (6.7.2)

Let $\underline{j''}_{r-1}=(\underline{j}_{r-1}, n)$ be the tuple corresponding to the reduced expression $w_{r-1}s_{n}=\prod\limits_{l_{1}=1}^{r-1}[a_{l_{1}}, n].$ Then
by Lemma \ref{lemma 7.2} the natural homomorphism 
\begin{center}
	$H^0(Z(\tau_{r}, \underline{j'}_{r}), T_{(\tau_{r}, \underline{j'}_{r})})\longrightarrow H^0(Z(w_{r-1}s_{n}, \underline{j''}_{r-1}), T_{(w_{r-1}s_{n}, \underline{j''}_{r-1})})$  
\end{center}
is surjective. By (6.7.2), we have  dim($H^0(Z(w_{r-1}s_{n}, \underline{j''}_{r-1}), T_{(w_{r-1}s_{n}, \underline{j''}_{r-1})_{\mu}})$)$\le2.$ \hspace{2cm}(6.7.3)

Again by the Lemma \ref{lemma 7.2} the natural homomorphism
\begin{center}
$H^0(Z(w_{r-1}s_{n}, \underline{j''}_{r-1}), T_{(w_{r-1}s_{n}, \underline{j''}_{r-1})})\longrightarrow H^0(Z(w_{r-1}, \underline{j}_{r-1}), T_{(w_{r-1}, \underline{j}_{r-1})})$ \hspace{2.2cm}(6.7.4)  
\end{center}
is surjective.

By LES we have the following long exact sequence of $B$-modules:
\begin{center}
	$0\longrightarrow H^0(w_{r-1}, \alpha_{n-1})\longrightarrow H^0(Z(w_{r-1}, \underline{j}_{r-1}), T_{(w_{r-1}, \underline{j}_{r-1})})\longrightarrow $

	$H^0(Z(\tau_{r-1}, \underline{j'}_{r-1}), T_{(\tau_{r-1}, \underline{j'}_{r-1})})\longrightarrow H^1(w_{r-1}, \alpha_{n-1})\longrightarrow \cdots$
\end{center}

Since $H^1(w_{r}, \alpha_{n-1})_{\mu}\neq0$, then by Corollary \ref {cor 6.8} we have $H^1(w_{r-1}, \alpha_{n-1})_{\mu}=0.$
Therefore we have an exact sequence

	$$0\longrightarrow H^0(w_{r-1}, \alpha_{n-1})_{\mu}\longrightarrow H^0(Z(w_{r-1}, \underline{j}_{r-1}), T_{(w_{r-1}, \underline{j}_{r-1})})_{\mu}\longrightarrow $$
$$	H^0(Z(\tau_{r-1}, \underline{j'}_{r-1}), T_{(\tau_{r-1}, \underline{j'}_{r-1})})_{\mu}\longrightarrow 0 \hspace{8.7cm} (6.7.5)$$

Let $\sigma_{r-1}=w_{r-2}s_{n}s_{a_{r-1}\cdots s_{n+2-r}},$ and let $\underline{j^*}_{r-1}=(\underline{j}_{r-2}, n, a_{r-1}, a_{r-1}+1,\dots , n+2-r)$ be the reduced expression of $\sigma_{r-1}.$

Then by using Lemma \ref {lemma 7.2} the natural homomorphism 
$$H^0(Z(\tau_{r-1}, \underline{j'}_{r-1}), T_{(\tau_{r-1}, \underline{j'}_{r-1})})\longrightarrow H^0(Z(\sigma_{r-1}, \underline{j^*}_{r-1}), T_{(\sigma_{r-1}, \underline{j^*}_{r-1})})\hspace{3.6cm}(6.7.6) $$ 
is surjective. Therefore, by (6.7.5) and (6.7.6) we have

$H^0(Z(w_{r-1}, \underline{j}_{r-1}), T_{(w_{r-1}, \underline{j}_{r-1})})_{\mu}\longrightarrow H^0(Z(\sigma_{r-1}, \underline{j^*}_{r-1}),   T_{(\sigma_{r-1}, \underline{j^*}_{r-1})})_{\mu} \hspace{2.8cm} (6.7.7)$ is surjective. Since $H^1(w_{r}, \alpha_{n-1})_{\mu}\neq0$, by Corollary \ref{cor 6.3} we have $H^0(\sigma_{r-1},  \alpha_{n+2-r})_{\mu}\neq0.$
Therefore, $H^0(Z(\sigma_{r-1}, \underline{j^*}_{r-1}),   T_{\sigma_{r-1}, \underline{j^*_{r-1}})})_{\mu}\neq0.$
 
Thus from (6.7.7) we have $H^0(Z(w_{r-1}, \underline{j}_{r-1}), T_{(w_{r-1}, \underline{j}_{r-1})})_{\mu}\neq0.$ Then we have an exact sequence of $T$-modules
\begin{center}
$0\longrightarrow H^0(w_{r-1}s_{n}, \alpha_{n})_{\mu}\longrightarrow H^0(Z(w_{r-1}s_{n}, \underline{j''}_{r-1}), T_{(w_{r-1}s_{n}, \underline{j''}_{r-1})})_{\mu}\longrightarrow H^0(Z(w_{r-1}, \underline{j}_{r-1}),   T_{(w_{r-1}, \underline{j}_{r-1})})_{\mu}\longrightarrow0 $
\end{center}

Since $H^1(w_{r}, \alpha_{n-1})_{\mu}\neq 0,$ by Corollary \ref {cor 6.4} we have $H^0(w_{r-1}s_{n}, \alpha_{n})_{\mu}\neq0.$

Therefore dim($H^0(Z(w_{r-1}s_{n}, \underline{j''}_{r-1}), T_{(w_{r-1}s_{n}, \underline{j''}_{r-1})_{\mu}})$)$\ge 2.$

Hence, by (6.7.3) we have dim($H^0(Z(w_{r-1}s_{n}, \underline{j''}_{r-1}), T_{(w_{r-1}s_{n}, \underline{j''}_{r-1})})_{\mu}$)$=2.$

Since the natural homomorphism 
\begin{center}
	$H^0(Z(\tau_{r}, \underline{j'}_{r}), T_{(\tau_{r}, \underline{j'}_{r})})\longrightarrow H^0(Z(w_{r-1}s_{n}, \underline{j''}_{r-1}), T_{(w_{r-1}s_{n}, \underline{j''}_{r-1})})$ 
\end{center}
is surjective, we have dim($H^0(Z(\tau_{r}, \underline{j'}_{r}), T_{(\tau_{r}, \underline{j'}_{r})})_{\mu}$)$=2.$      

Therefore by (6.7.1),
\begin{center}
	$H^0(Z(\tau_{r}, \underline{j'}_{r}), T_{(\tau_{r}, \underline{j'}_{r})})_{\mu}\longrightarrow H^1(w_{r}, \alpha_{n-1})_{\mu}$
\end{center}
is surjective.

\end{proof}
\end{lem}

\section{main theorem}
In this section we prove the main theorem.

Recall that $G=PSO(2n+1, \mathbb{C}) (n\ge 3),$ and let $c$ be a Coxeter element of $W.$ Then there exists a decreasing sequence 
$n\ge a_{1}> a_{2}>\dots >a_{k}= 1$ of positive integers such that $c=[a_{1}, n][a_{2}, a_{1}-1]\cdots[a_{k}, a_{k-1}-1],$ where $[i, j]$ for $i\le j$ denotes $s_{i}s_{i+1}\cdots s_{j}.$

Let $\underline{i}=(\underline{i}^1, \underline{i}^2 ,\dots ,\underline{i}^n)$ be a sequence corresponding to a reduced expression of $w_{0},$ where $\underline{i}^r$ ($1\le r\le n$) is a sequence of reduced expressions of $c$ (see Lemma \ref {lemma 3.4}).

\begin{thm}\label{theorem 8.1}
$H^j(Z(w_{0},\underline{i}), T_{(w_{0}, \underline{i})})=0$ for all $j\ge 1$ if and only if $a_{2}\neq n-1.$	
\end{thm}
\begin{proof}
From \cite[Proposition 3.1, p. 673]{CKP}, we have $H^j(Z(w_{0}, \underline{i}), T_{(w_{0}, \underline{i})})=0$ for all $j\ge 2.$ 
It is enough to prove the following:
$H^1(Z(w_{0}, \underline{i}), T_{(w_{0}, \underline{i})})=0$ if and only if $c$ is of the form $[a_{1}, n][a_{2}, a_{1}-1]\cdots[a_{k}, a_{k-1}-1]$ with $a_{2}\neq n-1.$

Proof of $(\implies)$: If  $a_{2}=n-1,$ then $a_{1}=n$ and $c=s_{n}s_{n-1}v,$ where $v\in W_{J}$ and $J=S\setminus \{\alpha_{n-1}, \alpha_{n}\}.$ Let $u=s_{n}s_{n-1}$. Then $c=uv.$
Let $\underline{j}=(n, n-1)$ be the sequence corresponding to $u.$ Then
using LES, we have: 

$$0\longrightarrow H^0(u,\alpha_{n-1})  \longrightarrow H^0(Z(u,\underline{j}), T_{(u,\underline{j})})\longrightarrow H^0(s_{n}, \alpha_{n})\longrightarrow$$

$$ H^1(u,\alpha_{n-1})\xrightarrow f H^1(Z(u,\underline{j}), T_{(u,\underline{j})})\longrightarrow 0.$$

We see that $H^1(s_{n}s_{n-1}, \alpha_{n-1})=\mathbb{C}_{\alpha_{n} + \alpha_{n-1}}$ and $H^0(s_{n}, \alpha_{n})_{\alpha_{n} + \alpha_{n-1}}=0.$ Hence $f$ is non zero homomorphism.
Hence $H^1(Z(u, \underline{j}), T_{(u,\underline{j})}))\neq 0.$ 
By Lemma \ref {lemma 7.1}, the natural homomorphism 

$$H^1(Z(w_{0}, \underline{i}), T_{(w_{0}, \underline{i})}) \longrightarrow H^1(Z(u,\underline{j}), T_{(u, \underline{j})})$$ is surjective.

Hence we have 

$$H^1(Z(w_{0}, \underline{i}), T_{(w_{0}, \underline{i})})\neq 0.$$

Proof of $(\impliedby):$ Assume that $a_{2}\neq n-1.$ Recall that $w_{k}=[a_{1}, n]\cdots [a_{k-1}, n][a_{k}, n-1],$ $u_{1}=w_{k}s_{n}[a_{k}, n-1].$ By Lemma \ref {lemma 7.3}(2), the natural homomorphism
\begin{center}
	$H^1(Z(w_{0}, \underline{i}), T_{(w_{0}, \underline{i})})\longrightarrow H^1(Z(u_{1}, \underline{i}_{1}),T_{(u_{1}, \underline{i}_{1})})$ \hspace{6.5cm}      (7.1.1)
\end{center}
of $B$-modules is an isomorphism.

 By using LES, we have the following exact sequence of $B$-modules:

	$$\cdots\longrightarrow H^0(Z(u_{1}, \underline{i}_{1}),T_{(u_{1}, \underline{i}_{1})})\longrightarrow H^0(Z(u'_{1}, \underline{i'}_{1}),T_{(u'_{1}, \underline{i'}_{1})})\xrightarrow {h_{1}}$$
	$$H^1(u_{1}, \alpha_{n-1})
	\longrightarrow H^1(Z(u_{1}, \underline{i}_{1}),T_{(u_{1}, \underline{i}_{1})})\longrightarrow H^1(Z(u'_{1}, \underline{i'}_{1}),T_{(u'_{1}, \underline{i'}_{1})})\longrightarrow0.$$

By Corollary \ref{cor 7.6}, we see that the natural homomorphism $h_{1}:H^0(Z(u'_{1}, \underline{i'}_{1}), T_{(u'_{1}, \underline{i'}_{1})})\longrightarrow H^1(u_{1}, \alpha_{n-1}) $
is surjective. Therefore, the natural homomorphism

$$H^1(Z(u_{1}, \underline{i}_{1}),T_{(u_{1}, \underline{i}_{1})})\longrightarrow H^1(Z(u'_{1}, \underline{i'}_{1}),T_{(u'_{1}, \underline{i'}_{1})})\hspace{6.1cm} (7.1.2)$$
is an isomorphism.

By Lemma \ref {lemma 7.2}(2), the natural homomorphism 

$$ H^1(Z(u'_{1}, \underline{i'}_{1}),T_{(u'_{1}, \underline{i'}_{1})})\longrightarrow H^1(Z(w_{k}, \underline{j}_{k}), T_{(w_{k}, \underline{j}_{k})})\hspace{5.9cm}  (7.1.3)$$

is an isomorphism.

Therefore, by  (7.1.1), (7.1.2) and (7.1.3) the natural homomorphism 

$$H^1(Z(w_{0}, \underline{i}), T_{(w_{0}, \underline{i})})\longrightarrow H^1(Z(w_{k}, \underline{j}_{k}), T_{(w_{k}, \underline{j}_{k})})     \hspace{6.1cm}      (7.1.4)$$

is an isomorphism. 

Recall that $\tau_{k}=[a_{1}, n][a_{2}, n]\cdots [a_{k-1}, n][a_{k}, n-2].$ By using LES, we have the following exact sequence of $B$-modules:

$$\cdots\longrightarrow H^0(Z(w_{k}, \underline{j}_{k}),T_{(w_{k}, \underline{j}_{k})})\longrightarrow H^0(Z(\tau_{k}, \underline{j'}_{k}),T_{(\tau_{k}, \underline{j'}_{k})})\xrightarrow {h_{2}}$$
$$H^1(w_{k}, \alpha_{n-1})
\longrightarrow H^1(Z(w_{k}, \underline{j}_{k}),T_{(w_{k}, \underline{j}_{k})})\xrightarrow{h_{3}} H^1(Z(\tau_{k}, \underline{j'}_{k}),T_{(\tau_{k}, \underline{j'}_{k})})\longrightarrow0.$$

By Lemma \ref {lemma 7.7}(2), we see that the map $h_{2}:H^0(Z(\tau_{k}, \underline{j'}_{k}),T_{(\tau_{k}, \underline{j'}_{k})})\longrightarrow H^1(w_{k}, \alpha_{n-1})$
is surjective.

 Therefore, the map $h_{3}: H^1(Z(w_{k}, \underline{j}_{k}),T_{(w_{k}, \underline{j}_{k})})\longrightarrow H^1(Z(\tau_{k}, \underline{j'}_{k}),T_{(\tau_{k}, \underline{j'}_{k})})\hspace{1cm} (7.1.5)$
 
  is an isomorphism.

By using Lemma \ref {lemma 7.2}(2) we see that the natural map

$$ H^1(Z(\tau_{k}, \underline{j'}_{k}),T_{(\tau_{k}, \underline{j'}_{k})})\longrightarrow H^1(Z(w_{k-1}, \underline{j}_{k-1}),T_{(w_{k-1}, \underline{j}_{k-1})})\hspace{4cm}(7.1.6)$$

is an isomorphism.

By using LES, we have the following exact sequence of $B$-modules:

$$\cdots\longrightarrow H^0(Z(w_{k-1}, \underline{j}_{k-1}),T_{(w_{k-1}, \underline{j}_{k-1})})\longrightarrow H^0(Z(\tau_{k-1}, \underline{j'}_{k-1}),T_{(\tau_{k-1}, \underline{j'}_{k-1})})\rightarrow $$
$$H^1(w_{k-1}, \alpha_{n-1})
\longrightarrow H^1(Z(w_{k-1}, \underline{j}_{k-1}),T_{(w_{k-1}, \underline{j}_{k-1})})\rightarrow H^1(Z(\tau_{k-1}, \underline{j'}_{k-1}),T_{(\tau_{k-1}, \underline{j'}_{k-1})})\longrightarrow0.$$

By Lemma \ref {lemma 7.7}(2), we see that the natural map $H^0(Z(\tau_{k-1}, \underline{j'}_{k-1}),T_{(\tau_{k-1}, \underline{j'}_{k-1})})\longrightarrow H^1(w_{k-1}, \alpha_{n-1})$
is surjective.

 Therefore, the natural map 
 
 $ H^1(Z(w_{k-1}, \underline{j}_{k-1}),T_{(w_{k-1}, \underline{j}_{k-1})})\longrightarrow H^1(Z(\tau_{k-1}, \underline{j'}_{k-1}),T_{(\tau_{k-1}, \underline{j'}_{k-1})})$ \hspace{3cm} (7.1.7)
  is an isomorphism.

By using Lemma \ref {lemma 7.2}(2) and \ref {lemma 7.7}(2) repeatedly, we see that the natural map

$$ H^1(Z(\tau_{k-1}, \underline{j'}_{k-1}),T_{(\tau_{k-1}, \underline{j'}_{k-1})})\longrightarrow H^1(Z(\tau_{r}, \underline{j'}_{r}),T_{(\tau_{r}, \underline{j'}_{r})})\hspace{5cm}  (7.1.8)$$

is an isomorphism for all $3\le r\le k-2.$

Therefore by ( (7.1.4), (7.1.5), (7.1.6), (7.1.7), (7.1.8) we have the natural homomorphism

$$H^1(Z(w_{0}, \underline{i}), T_{(w_{0}, \underline{i})})\longrightarrow H^1(Z(\tau_{3}, \underline{j'}_{3}), T_{(\tau_{3}, \underline{j'}_{3})})      \hspace{7cm}    (7.1.9)$$

is an isomorphism.

Again from Lemma \ref {lemma 7.2}(2), we see that the natural map

$$ H^1(Z(\tau_{3}, \underline{j'}_{3}),T_{(\tau_{3}, \underline{j'}_{3})})\longrightarrow H^1(Z(w_{2}, \underline{j}_{2}),T_{(w_{2}, \underline{j}_{2})})$$

is an isomorphism.

Since $a_{2}\neq n-1,$ by Lemma \ref{lemma 6.2}(2) we have $H^1(w_{2}, \alpha_{n-1})=0.$ Note that by Lemma \ref {lemma 6.2}(1) we have $H^1(w_{1}, \alpha_{n-1})=0.$ 

By using Lemma \ref {lemma 2.5} and using LES, we have  $H^1(Z(w_{2}, \underline{j}_{2}),T_{(w_{2}, \underline{j}_{2})})=0.$ Hence we conclude that $H^1(Z(w_{0}, \underline{i}), T_{(w_{0}, \underline{i})})=0.$ This completes the proof of the theorem.

\end{proof}
\begin{cor}
	Let $c$ be a Coxeter element such that $c$ is of the form $[a_{1}, n][a_{2}, a_{1}-1]\cdots[a_{k}, a_{k-1}-1]$ with $a_{2}\neq n-1$ and $a_{k}=1.$
	Let $(w_{0}, \underline{i})$ be a reduced expression of $w_{0}$
	in terms of $c$ as in Theorem \ref {theorem 8.1}. Then, $Z(w_{0}, \underline{i})$ has no deformations.
\end{cor}
\begin{proof}
	By Theorem \ref {theorem 8.1} and by \cite[Proposition 3.1, p.673]{CKP}, we have $H^i(Z(w_{0}, \underline{i}), T_{(w_{0}, \underline{i})})=0$ for all $i>0.$ Hence, by \cite[Proposition 6.2.10, p.272]{Huy}, we see that $Z(w_{0}, \underline{i})$ has no deformations.
\end{proof}

\begin{rem}
Theorem 7.1 does not hold for $PSO(5, \mathbb{C}).$ 
\end{rem}
\begin{proof} We take $c=s_{1}s_{2}.$ Here $a_{2} \neq 1.$ 
Further, we have $w_{0}=c^2=s_{1}s_{2}s_{1}s_{2}.$ Let $\underline{i}=(1, 2, 1, 2).$ It is easy to see by using SES repeatedly that $H^1(s_{1}s_{2}s_{1}, \alpha_{1})=\mathbb{C}_{\alpha_{1}+\alpha_{2}}\oplus \mathbb{C}_{\alpha_{2}}.$ Further, note that $H^0(Z(s_{1}s_{2},(1,2)),T_{(s_{1}s_{2},(1,2))})_{\alpha_{1}+\alpha_{2}}=0$ (see \cite[Proposition 6.3(1), p.688]{CKP}). Hence by using LES we have

$H^1(Z(s_{1}s_{2}s_{1},(1,2,1)),T_{(s_{1}s_{2}s_{1},(1,2,1))})\neq0.$

Therefore by using Lemma \ref{lemma 7.1} we have $H^1(Z(w_{0}, \underline{i}), T_{(w_{0}, \underline{i})})\neq 0.$ 
\end{proof}

{\bf Acknowledgements} 
We thank the referee for the useful comments and suggestions. We thank the 
Infosys Foundation for the partial finance support.

\end{document}